%% file: ruledhypArXiv.tex
\documentclass[a4paper,10pt]{amsart}
\usepackage[]{amsmath}
 \usepackage{amsfonts,mathrsfs}
 \usepackage{amsthm}
 \usepackage{amssymb}
 \usepackage{enumerate}
 \usepackage{tikz-cd}
 \usepackage{cleveref}

\include{operators}
\include{fonts}

\theoremstyle{definition}
\newtheorem{Def}{Definition}[section]
\newtheorem{ex}[Def]{Example}
\newtheorem{rem}[Def]{Remark}

\theoremstyle{plain}
\newtheorem{prop}[Def]{Proposition}
\newtheorem{thm}[Def]{Theorem}
\newtheorem{facts}[Def]{Facts}
\newtheorem*{thm*}{Theorem}
\newtheorem{lem}[Def]{Lemma}
\newtheorem{cor}[Def]{Corollary}
\newtheorem*{cor*}{Corollary}

\newtheorem*{con*}{Conjecture}
\newtheorem*{frag*}{Question}
\newtheorem*{verm*}{Vermutung}

\usepackage{mathrsfs}

\title[Hyperbolic Secant Varieties of $M$-Curves]{Hyperbolic Secant Varieties of $M$-Curves}
\author{Mario Kummer}
\address{Technische Universit\"at Berlin, Fakult\"at II -- Mathematik und Naturwissenschaften, Institut f\"ur Mathematik, Berlin, Germany\\
ORCiD: 0000-0001-6378-4808}
\email{kummer@tu-berlin.de}
\author{Rainer Sinn}
\address{Freie Universit\"at Berlin, Fachbereich Mathematik und Informatik, Institut f\"ur Mathematik, Berlin, Germany\\
ORCiD: 0000-0001-9456-6889}
\email{rsinn@zedat.fu-berlin.de}
\thanks{The authors have been supported by the Deutsche Foschungsgemeinschaft under Grants No. 421473641 and 426054364.}
\newcommand\sL{\mathscr{L}}
\newcommand\sV{\mathcal{V}}
\newcommand\wh{\widehat}
\newcommand\Sec{\sigma}
\newcommand\ol{\overline}
\DeclareMathOperator{\supp}{supp}

\newcommand{\placeholderone}{vastly real}
\newcommand{\placeholdertwo}{maximally odd}

\begin{document}

\subjclass[2010]{Primary: 14M12, 14P05, 14P25; Secondary: 90C22}

\begin{abstract}
  We relate the geometry of curves to the notion of hyperbolicity in real algebraic geometry. A hyperbolic variety is a real algebraic variety that (in particular) admits a real fibered morphism to a projective space whose dimension is equal to the dimension of the variety. We study hyperbolic varieties with a special interest in the case of hypersurfaces that admit a real algebraic ruling. The central part of the paper is concerned with secant varieties of real algebraic curves where the real locus has the maximal number of connected components, which is determined by the genus of the curve. For elliptic normal curves, we further obtain definite symmetric determinantal representations for the hyperbolic secant hypersurfaces, which implies the existence of symmetric Ulrich sheaves of rank one on these hypersurfaces.
\end{abstract}
\maketitle

\section{Introduction}
Hyperbolic polynomials and the associated algebraic hypersurfaces are at the intersection of algebraic geometry, optimization, combinatorics, and computer science.
Historically, they go back to partial differential equations in the work of Peter Lax \cite{Lax58} and have started a new life in convex programming since the advent of interior point methods \cite{Gu97,RenMR2198215}. In real algebraic geometry, (smooth) hyperbolic hypersurfaces are an extremal topological type, namely the most nested \cite{HV07,victorsurvey}.

In this paper, we provide a geometric construction of highly singular hyperbolic hypersurfaces, which is interesting from several perspectives. 
The hyperbolicity of secant varieties of $M$-curves is closely linked to a property of linear systems on real algebraic curves that we call \emph{\placeholderone}. This is a common generalization of Ahlfors's circle maps \cite{Ahl50, gabard}, i.e. real fibered (also called separating or totally real) morphisms to the projective line,
and Mikhalkin and Orevkov's maximally writhed links \cite{mikhalkin2019maximally,mikhalkin2019rigid}. Classically, the focus in the geometry of real algebraic curves is on curves in the projective plane or $3$-space. Vastly real linear systems on curves are an extremal real embedding into projective spaces of almost any odd dimension. 
In addition to this generalization, certain vastly real embeddings provide a new totally real instance in enumerative algebraic geometry concerning special secant planes. 

In classical algebraic geometry, an attractive class of examples of hyperbolic hypersurfaces are definite symmetroids: hypersurfaces whose defining polynomials are the determinant of real symmetric matrix pencils that contain a definite matrix. Symmetric determinantal representations imply the existence of a symmetric Ulrich sheaf of rank one on the hypersurface \cite{beauvilleUlrich}. We show the existence of such determinantal representations for secant varieties of elliptic normal $M$-curves.

From the point of view of optimization, highly singular hyperbolicity cones are necessary to even have a chance to construct expressive hierarchies analogous to sum-of-squares and moment methods widely used in semidefinite programming (as shown in the work of Saunderson \cite{saunderson2019limitations} building on \cite{averkov2019optimal}). Our geometric construction provides such examples. It can be used to write certain convex hulls of connected components of real curves as projections of hyperbolicity cones.
The hyperbolic secant hypersurfaces for curves of genus at least $2$ are a promising testing ground for the Generalized Lax-Conjecture, which claims that they admit definite determinantal representations up to a cofactor with controlled hyperbolicity cone \cite{victorsurvey}. 

We discuss the different points of view and motivation for our work in more detail.
The main idea of our work is to study highly singular hyperbolic hypersurfaces by looking for tractable ruled examples. By this, we mean irreducible hypersurfaces that contain a Zariski-dense union of linear spaces of a fixed dimension that vary in an irreducible algebraic family.

In \Cref{sec:onedimrule}, we study the simplest case: A hypersurface in $\P^n$ ruled by a $1$-dimensional family of $(n-2)$-dimensional linear spaces. The main result in this section,  \Cref{thm:onedimrule}, says that this is only possible in the trivial way, namely by taking cones over hyperbolic plane curves. This implies the Generalized Lax Conjecture for convex hulls of curves in three-space (\Cref{cor:hypisspec}).

The following \Cref{sec:hypsecMcurve} is devoted to the question when the secant varieties of $X\subset \P^n$ are hyperbolic. We restrict our attention to the case that $X$ is a smooth irreducible $M$-curve, i.e.~an irreducible real algebraic curve of genus $g$ such that $X(\R)$ has $g+1$ connected components. In this case, \Cref{thm:secantchara}  gives a characterization of hyperbolicity of the secant varieties in terms of the linear system embedding $X$. Hyperbolicity of the secant varieties corresponds to linear systems $\sL$ that are \emph{\placeholderone} in the sense that the support of each divisor in $\sL$ contains many real points (see \Cref{def:vastlyreal}).
Such linear systems always exist on $M$-curves by \Cref{thm:tollimpliesveryreal} (with sufficiently high degree relative to the genus). This existence proof rests on the fact that linear combinations of \emph{interlacing} sections have many real roots.
On the other hand, it remains an open question whether such linear systems of dimension at least $5$ exist on real curves that are not $M$-curves.
Along the way, by studying {\placeholderone} linear systems on $M$-curves, we find a new totally real instance in enumerative geometry. Given a general curve $C\subset \P^{2k}$ for a positive integer $k$, we are interested in the $(k-1)$-dimensional planes $\Lambda\subset \P^n$ that intersect $C$ in $k+1$ points. By genericity, there is a finite number of such $(k-1)$-planes that has been counted in \cite{ACGHI}. Using the hyperbolicity of the appropriate secant variety, we can show that there are totally real instances of this enumerative problem (\Cref{thm:totallyreal}). 
If the secant variety in question happens to be a hypersurface, its defining polynomial is hyperbolic and there is an associated convex cone, its hyperbolicity cone. The structure as a convex cone is simple: every proper face is a simplex, see \Cref{cor:simplicial}.

In \Cref{sec:elliptic}, we focus on the case of curves of genus $1$ that are embedded by a complete linear system (usually called elliptic normal curves). Compared to curves of higher genus, many things about the secants of these curves are much easier to understand: The secant varieties are smooth outside the lower secant varieties, their vanishing ideals and their resolutions can be understood. All these problems in classical algebraic geometry are much harder for curves of higher genus. It is the same from our point of view: We use the detailed knowledge about elliptic normal curves to construct definite determinantal representations of secant hypersurfaces to elliptic normal $M$-curves, see \Cref{thm:ellnormaldetrep}. The main line of the construction is similar to Hesse's construction of determinantal representations for plane quartics often called Dixon's process. Concretely, we find a \emph{contact interlacer} for these hyperbolic hypersurfaces, which then determines the determinantal representation. To show existence of this contact interlacer (and to do necessary dimension counts for the following construction), we focus on the symmetric products of elliptic curves and its relation to the secant varieties. This approach again only works out so nicely for curves of genus $1$. 

In \Cref{sec:extendedform}, we study in particular, how these results can be used to write convex hulls of curves as projections of hyperbolicity cones. In general, we can write convex hulls of connected components of $M$-curves as projections of hyperbolicity cones whose degrees and dimensions can be bounded in terms of degree and genus of the curves, see \Cref{thm:hyperbolicshadows}. In case of elliptic curves, using the results from the previous section, we obtain semidefinite representations for their convex hulls of uniform size (i.e.~they only depend on the degree of the curve). This improves previous results obtained by Scheiderer \cite{clausg1, clauscurves} using sum-of-squares methods.

We conclude the paper with a short section presenting open questions.

\section{Preliminaries and Definitions}
\subsection{Real algebraic geometry}
Unless otherwise stated, we work in the category of separated schemes of finite type over $\R$ and call an object a variety if it is in addition reduced. For such a variety $X$ and $\mathbb{K}=\R$ or $\C$, we will write $X(\mathbb{K}) = \Hom_{\Spec \R}(\Spec \mathbb{K}, X)$ for the set of $\mathbb{K}$-points. We call $X$ a curve if it is pure of dimension one. By $\pp^n=\pp^n_\R=\textrm{Proj}(\R[x_0,\ldots,x_n])$ we denote the projective space over $\R$. A semi-algebraic subset of $X(\R)$, where $X$ is an affine variety, is a finite boolean combination of subsets of the form $\{x\in X(\R):\, f(x)>0\}$ for some regular function $f$ on $X$. By glueing, the definition carries over to arbitrary varieties. For an introductory reference to real algebraic geometry we refer to \cite{BCR}.

\subsection{Convex geometry}
We will study convexity in the projective setup: We call a subset $S\subset\pp^n(\R)$ \emph{convex} if it is of the form $\pp(K)$ for some convex set $K\subset\R^{n+1}$, i.e. the image of $K$ under the projection map $\R^{n+1}\setminus\{0\}\to\pp^n(\R),\,x\mapsto [x]$. In this setting, taking the convex hull comes with an implicit choice.

\begin{Def}$\,$
  \begin{enumerate}
    \item A semi-algebraic set $S\subset \P^n(\R)$ is \emph{very compact} if there is a real hyperplane $H\subset \P^n$ such that $H\cap S = \emptyset$. 
    \item If $S\subset \P^n(\R)$ is a connected, semi-algebraic, and very compact set, then the \emph{convex hull} of $S$ is defined as $\conv(S) = \iota_H(\conv(\iota_H^{-1}(S)))\subset \P^n$, where
    \[
      \iota_H\colon \R^n = (\P^n\setminus H)(\R) \to \P^n(\R).
    \]
  \end{enumerate}
\end{Def}

The above definition of $\conv(S)$ does not depend on the choice of hyperplane $H$ with $S\cap H = \emptyset$, as the following lemma shows.
\begin{lem}
  Let $S\subset\P^n(\R)$ be a connected, semi-algebraic, very compact set. 
  Let $\ell\in(\P^{n})^*(\R)$ with $\cV(\ell)\cap S = \emptyset$. We have $\conv(S) = \P(\cone \{x\in \widehat{S}\colon \ell(x) > 0\})$. Here $\widehat{S}$ denotes the affine cone over $S$.
  In particular, $\conv(S)$ does not depend on the choice of $\ell$ with $\cV(\ell)\cap S = \emptyset$. 
\end{lem}

\begin{proof}
  We assume, after a change of coordinates, if necessary, that we form $\conv(S)$ with respect to $H = \{(x_0:x_1:\ldots:x_n)\in\P^n\colon x_0 = 0\}$. 
  A point $v\in \conv(S)$ is of the form $[\sum_{i}\lambda_i v_i]$ with $\lambda_i \in \R_{\geq 0}$, $\sum_i \lambda_i = 1$, and $v_i\in \iota^{-1}_H(S)$. Since $\ell(1,v_i)\neq 0$, either $(1,v_i)$ or $(-1,-v_i)$ is in $\{x\in \widehat{S}\colon \ell(x)>0\}$. So $\iota_H(v)$ is in $\P(\cone \{x\in \widehat{S}\colon \ell(x)>0\})$. Conversely, a point in $\P(\cone \{x\in \widehat{S}\colon \ell(x)>0\})$ is of the form $\sum_{i}\lambda_i v_i$ with $v_i \in \widehat{S}$, $\ell(v_i)>0$, and $\lambda\in \R_{\geq 0}$. Since $S\subset \P^n(\R)$ is connected, either the first entry of every $v_i$ is positive or it is negative. Suppose we are in the first case. Then we can rescale every $v_i$ such that its first entry is $1$. By appropriately rescaling the coefficient $\lambda_i$, the sum $\sum_i \lambda_i v_i$ still gives the same point. 
  By rescaling the entire conic combination, we can assume that $\sum_i \lambda_i = 1$. So the point $\sum_i \lambda_i v_i$ lies on the line spanned by a point in $\iota_H(\conv(\iota_H^{-1}(S)))$. In case that the first entry of every $v_i$ is negative, we argue similarly by rescaling this first entry to be $-1$. 

  Now that we have established the equality of the sets $\iota_H(\conv(\iota_H^{-1}(S)))$ and $\P(\cone\{x\in \widehat{S}\colon \ell(x)>0\})$, it becomes clear that the definition does not depend on the choice of the linear form $\ell$ as long as the hyperplane $\{x\in \P^n\colon \ell(x) = 0\}$ does not intersect $S$. Indeed, the set $\{x\in \widehat{S}\colon \ell(x)>0\}$ is one of the two connected components of $\widehat{S}\setminus\{0\}$ that are reflections of each other.
\end{proof}

\subsection{Hyperbolic polynomials and varieties}
A morphism $f:X\to Y$ is \emph{real fibered} if $f(x)\in Y(\R) \Leftrightarrow x\in X(\R)$ for all $x\in X$.
A closed subvariety $X\subset\pp^n$ of dimension $k$ is called \emph{hyperbolic} with respect to a linear subspace $E\subset\pp^n$ of codimension $k+1$ if $E\cap X=\emptyset$ and the linear projection $\pi_E:X\to\pp^k$ with center $E$ is real fibered. See \cite{kum20,Sha18} for more about hyperbolic varieties and real fibered morphisms. A homogeneous polynomial $h\in\R[x_0,\ldots,x_n]$ is hyperbolic with respect to a point $0\neq e\in\R^{n+1}$ if its zero set in $\pp^n$ is hyperbolic with respect to $[e]\in\pp^n$. This is equivalent to saying $h(e)\neq0$ and $h(te-v)$ has only real zeros for all $v\in\R^{n+1}$. Note that the latter is the original definition of hyperbolic polynomials. The (closed) \emph{hyperbolicity cone} of $h$ with respect to $e$ is the set of all $v\in\R^{n+1}$ such that $h(te-v)$ has only nonnegative zeros. Equivalently, this is the closure of the connected component of $\{x\in\R^{n+1}:\,h(x)\neq0\}$ that containes $e$. The hyperbolicity cone is a closed convex cone, see \cite{gar}. Let $f,g\in\R[x_0,\ldots,x_n]$ with $d=\deg(g)+1=\deg(f)$ be hyperbolic with respect to $e$. If for all $v\in\R^{n+1}$ we have that $$a_1\leq b_1\leq a_2\leq\cdots\leq b_{d-1}\leq a_d$$ where the $a_i$ and $b_i$ are the zeros of $f(te-v)$ and $g(te-v)$ respectively, we say that $g$ \emph{interlaces} $f$, or that $g$ is an \emph{interlacer} of $f$. This definition carries over to hyperbolic hypersurfaces in the obvious way. For a detailed study of interlacers in the context of hyperbolic polynomials see \cite{MR3395549}.
If $X$ is a hypersurface that is hyperbolic with respect to $[e]$ and $K$ the hyperbolicity cone with respect to $e$ of its defining polynomial, then we say that $\pp(K)$ is the hyperbolicity cone of $X$ with respect to $[e]$. If $X$ is irreducible, then $\pp(K)$ does not depend on the choice of $e$ \cite{onecone} and we call it simply \emph{the hyperbolicity cone of} $X$. We say that a subset of $\pp^n(\R)$ is a \emph{hyperbolic shadow} if it is the image of a hyperbolicity cone under some linear projection $\pp^N\ratto\pp^n$.

\subsection{Spectrahedra}
A basic class of examples for hyperbolic polynomials is the following.
 Let $A_0,\ldots,A_n$ be some real symmetric matrices of size $d$ and for $x\in\R^{n+1}$ denote $A(x)=x_0A_0+\ldots+x_nA_n$. Assume that for some $e\in\R^{n+1}$ the matrix $A(e)$ is positive definite. Then the polynomial $h=\det A(x)$ is hyperbolic with respect to $e$: This follows after writing $A(e)$ as $B^tB$ for some invertible matrix $B$ from the fact that the characteristic polynomial of a real symmetric matrix has only real zeros. In this case that $h = \det(A(x))$, we say that $h$ has a \emph{definite determinantal representation}. The hyperbolicity cone of $h$ with respect to $e$ is the set of all $v\in\R^{n+1}$ for which $A(v)$ is positive semidefinite. Such sets are called \emph{spectrahedral cones}.  A \emph{spectrahedron} in $\R^n$ is an affine slice of a spectrahedral cone. We call a subset of $\pp^n(\R)$ a spectrahedron, if it is of the form $\pp(S)$ for some spectrahedron in $\R^{n+1}$. While every spectrahedral cone is a hyperbolicity cone, it is an open problem whether the converse is also true. This is the context of the Generalized Lax Conjecture. See \cite{HV07, bez} for some positive and \cite{Bra11} for some negative evidence for this conjecture. An overview is given in \cite{victorsurvey}. Finally, a \emph{spectrahedral shadow} is the image of a spectrahedron under some linear projection $\pp^N\ratto\pp^n$. Spectrahedral shadows amount to a large class of convex semi-algebraic sets \cite{MR2515796} but not every convex semi-algebraic set is a spectrahedral shadow \cite{clausschatten}.

\section{One dimensional ruling}\label{sec:onedimrule}
In this section, we are interested in hypersurfaces in $\P^n$ that are ruled by a $1$-dimensional family of linear subspaces, which therefore must have dimension $n-2$. After some general facts about varieties ruled by $1$-dimensional families, we characterize when such hypersurfaces are hyperbolic in terms of the ruling and show that the hyperbolic examples are trivial in the sense that they are in fact cones over hyperbolic plane curves.

\begin{Def}
  The \emph{center} of a projective algebraic variety $X\subset \P^n$, denoted by $\cent(X)$, is the set of all points $p\in \P^n$ such that the cone over $X$ with apex $p$ is contained in $X$. In other words, $p\in X$ and every line spanned by $p$ and some point $x\in X\setminus\{p\}$ is also contained in $X$.
\end{Def}

\begin{prop}
  Let $C\subset\G(k,n)$ be an irreducible curve. Then $X=\bigcup_{[\Lambda]\in C} \Lambda \subset \pp^n$ is irreducible of dimension $k+1$.
\end{prop}

\begin{proof}
  Consider the incidence correspondence
  \[
    \Phi_C = \Phi \times_{\G(k,n)} C =  \{(x,[\Lambda])\in \pp^n\times C \colon x\in\Lambda\}\subset \pp^n\times \G(k,n),
  \]
  where $\Phi$ is the universal $k$-plane (see \cite[\S 10.2]{3264}).
  The projection $\pi_2\colon \Phi_C \to C$ onto the second factor has irreducible fibers $\pi_2^{-1}([\Lambda]) = \Lambda\cong \pp^k$. It follows that $\Phi_C$ is irreducible of dimension $k+1$ (\cite[Book~1, Section~6.3]{MR1328833}).
  Therefore, the projection $\pi_1(\Phi_C)$ on the first factor, which is $X$, is irreducible.

  To determine the dimension, consider the fiber $\pi_1^{-1}(x)$ for a general point $x\in X$. If the dimension is $1$, i.e. $\dim(X)=k$, it means that every linear space $[\Lambda]\in C$ contains $x$ because $C$ is irreducible. So $X$ is a cone and $x$ is in the center of $X$. A general point cannot lie in the center unless $X$ is a linear space. Then $X$ is itself a $k$-plane, which is a contradiction to $C$ being a curve.
\end{proof}

So in particular, if $C\subset\G(n-2,n)$ is an irreducible curve, then $X=\bigcup_{[\Lambda]\in C} \Lambda \subset \pp^n$ is an irreducible hypersurface. 

\begin{lem}\label{lem:uniqueplane}
  Let $C\subset \G(n-2,n)$ be an irreducible curve. If $X = \bigcup_{[\Lambda]\in C} \Lambda$ is not a linear space, then a generic point on $X$ lies on a unique $(n-2)$-plane $[\Lambda]\in C$.
\end{lem}
 
\begin{proof}
  Assume that a generic point on $X$ lies on more than one $(n-2)$-plane from $C$. Then a generic point on a generic $(n-2)$-plane from $C$ will lie on more than one such $(n-2)$-plane. Since this generic $(n-2)$-plane $\Lambda\subset \P^n$ has to be swept out by the intersections $\Lambda\cap \Lambda'$ with $[\Lambda']\in C$, a dimension count shows that a generic $[\Lambda]\in C$ intersects a generic $[\Lambda']\in C$ in a plane of dimension at least $(n-3)$. Since this is a closed condition, we conclude that every $\Lambda$ intersects every other $\Lambda'\in C$ in a $(n-3)$-plane. But this is only possible if all $(n-2)$-spaces in $C$ lie in a common $(n-1)$-space which contradicts the assumption that $X$ is not a linear space.
\end{proof}

\begin{rem}
  In \Cref{lem:uniqueplane}, the assumption that $X$ is not a linear space is necessary: Consider for example the set $C\subset \G(1,2) = (\P^2)^\ast$ of tangents to a plane conic. Then $X = \bigcup_{[\Lambda]\in C}\Lambda$ is just the plane $\P^2$ spanned by the conic and a generic point in $X$ lies on two tangents.
\end{rem}

In the rest of this section, we assume that $X = \bigcup_{[\Lambda]\in C} \Lambda$ for an irreducible curve $C\subset \G(n-2,n)$ and that $X$ is not a linear space. We use the following notation:
We consider $\G(n-2,n)\subset\pp^{N}$ via the Pl\"ucker embedding, where $N=\binom{n+1}{2}-1$.
Fix a point $e\in\pp^n\setminus X$. Let $\alpha^1(e)=\{L\in\G(1,n):\, e\in L\}$ and $\alpha^{n-2}(e)=\{L\in\G(n-2,n):\, e\in L\}$.
We have $\dim \alpha^1(e)=n-1$ and $\dim \alpha^{n-2}(e)=2n-4$. To see this, consider the projection $\pi_e\colon \pp^n \ratto \pp^{n-1}$ away from $e$. This shows that $\alpha^1(e) \cong \pp^{n-1}$ and $\alpha^{n-2}(e) \cong \G(n-3,n-1)$.

\begin{lem}\label{lem:normalvectors}
  For every $L\in\G(1,n)$, the set $W_L=\{L'\in\G(n-2,n):\, L'\cap L\neq\emptyset\}$ has codimension $1$ in $\G(n-2,n)$ and is cut out by a single hyperplane $H_L\subset\pp^N$ with respect to the Pl\"ucker embedding.
  The set $\fL_e$ of all hyperplanes $H_L$ as $L$ varies over $\alpha^1(e)$ is a linear subspace of $(\pp^N)^\vee$ of dimension $n-1$.
\end{lem}

\begin{proof}
  We write the normal vector of the hyperplane $H_L\subset\P^N$ in terms of the Pl\"ucker coordinates of $L$. We represent $L$ by a $2\times(n+1)$ matrix $A$ of rank $2$, whose first row is equal to $e$ (which we can assume because of $e\in L$). The condition that a linear subspace $L'$ of $\P^n$ of dimension $n-2$ represented by the $(n-1)\times(n+1)$-matrix $A'$ intersects the line $L\subset \P^n$ says that the $(n+1)\times (n+1)$ matrix obtained from appending $A$ to $A'$ does not have full rank, i.e.~its determinant is $0$. This determinant is linear in the maximal minors of $A'$ (the Pl\"ucker coordinates of $L'$). The coefficients of these minors are the complementary maximal minors of $A$, which are linear in the second row of $A$. Therefore, $\fL_e$ is a linear subspace of $(\P^N)^\vee$. To compute its dimension, fix a hyperplane $U\subset \P^n$ that does not contain $e$ so that every choice of a point $x\in U$ defines the unique line $L_x$ through $e$ and $x$. The above matrix construction gives a linear map from $U\cong \P^{n-1}$ to $(\P^N)^\vee$, which has full rank. Indeed, the sets $W_{L_1}$ and $W_{L_2}$ are distinct if $L_1\neq L_2$ and so $H_{L_1}$ and $H_{L_2}$ are distinct too. 
\end{proof}

By construction, it is immediate that $\alpha^{n-2}(e)\subset W_L$ for all $L\in \alpha^1(e)$.
\begin{prop}\label{prop:hyp}
 The hypersurface $X = \bigcup_{[\Lambda]\in C} \Lambda$ is hyperbolic with respect to $e$ if and only if $H\cap C\subset C(\R)$ for all $H\in\fL_e(\R)$ with $C\not\subset H$. In particular, the curve $C$ has a smooth real point in that case.
\end{prop}

\begin{proof}
  On the one hand, in order to prove hyperbolicity with respect to $e$, it is sufficient to show that $L\cap X\subset X(\R)$ for generic lines $L\in\alpha^1(e)$. On the other hand, it suffices to check that $H\cap C\subset C(\R)$ for generic $H$ in order to conclude that $H\cap C\subset C(\R)$ for all $H$ with $C\not\subset H$.
  So to prove the claim, we show that, for generic lines $L\in \alpha^1(e)$, the intersection points of $L$ with $X$ are in one-to-one correspondence with the intersection points of $H_L$ with $C$ and that this bijection preserves reality.
  A point $[\Lambda]\in H_L\cap C$ is an $(n-2)$-plane $\Lambda\subset X$ that intersects $L\ni e$. So it gives an intersection point of $L$ with $X$.
  If $[\Lambda]\in H_L\cap C(\R)$, then $\Lambda$ and $L$ intersect in a real point by Gaussian elimination. Conversely, by \Cref{lem:uniqueplane} and genericity of the line $L$, every point $p\in L\cap X$ lies in a unique $(n-2)$-plane $\Lambda\subset X$. So $[\Lambda]\in H_L\cap C$. If the point $p\in X(\R)$ is real, then $[\Lambda]\in C(\R)$ because a general real point of $X$ lies on an $(n-2)$-plane $\Lambda$ with $\dim_\R(\Lambda(\R)) = n-2$.
\end{proof}

\begin{prop}\label{prop:cone}
  If $C\subset \Gr(n-2,n)$ is contained in $W_L$ for some $L\in \alpha^1(e)$, then all $(n-2)$-planes $[\Lambda]\in C$ share a common point. In particular, 
  if the linear system $\fL_e$ restricted to $C$ has (projective) dimension smaller than $n-1$, then all $(n-2)$-planes in $C$ share a point.
\end{prop}

\begin{proof}
  The intersection $X\cap L$ is finite since $e\not\in X$. On the other hand, because every $[L']\in C$ intersects $L$, the curve $C$ is covered by the closed sets $Y_p$ of all $(n-2)$-planes containing the point $p$ for $p\in X\cap L$. Since $C$ is irreducible, it is already contained in one of these sets $Y_p$ which implies the claim.
  The second part follows by Lemma~\ref{lem:normalvectors} from the first. 
\end{proof}

We will need the following standard fact below.
\begin{prop}\label{prop:center}
  Let $Y\subset \P^n$ be an irreducible variety with center $\cent(Y)$ and let $H\subset \P^n$ be a hyperplane with $\cent(Y)\not\subset H$. The center of the intersection $Y\cap H$ is the intersection $\cent(Y)\cap H$.
\end{prop}

\begin{proof}
  The inclusion that $\cent(Y)\cap H$ is contained in the center of $Y\cap H$ is direct.
  Let $q$ be a point in the center of $Y\cap H$ and $p$ a point in the center of $Y$, which is not in $H$. Let $x$ be a point on $Y$, $x\neq p,q$.
  The line connecting $p$ and $x$ intersects $H$ in a unique point $y$, which is also in $Y$. Since the plane spanned by $p,q,x$ is the same as that spanned by $p,q,y$, the facts that $p$ is in the center of $Y$ and that $q$ in the center of $Y\cap H$ imply that this plane is contained in $Y$. This shows that $q$ is in the center of $Y$.
\end{proof}

\begin{prop}\label{prop:noline}
  Let $Y\subset \P^n$ be a real curve such that every hyperplane intersects $Y$ only in real points. Then $Y$ is a line.
\end{prop}

\begin{proof}
  If $Y$ is not a line, we can find a non-real point $x\in Y$ such that the real line spanned by $x$ and its complex conjugate $\overline{x}$ is not equal to $Y$. This line is contained in a hyperplane that intersects $Y$ in at least two non-real points.
\end{proof}

\begin{thm}\label{thm:onedimrule}
  Let $C\subset \Gr(n-2,n)$ be an irreducible real curve and $X = \bigcup_{[\Lambda]\in C} \Lambda\subset \P^n$. 
  If $X$ is hyperbolic, then it is the cone over a plane hyperbolic curve. In this case its hyperbolicity cone is a spectrahedron.
\end{thm}

\begin{proof}
  Since the claim holds for hyperplanes, we can assume that $X$ is not a linear space. First we show that if $n>2$, then all $(n-2)$-planes that correspond to points in $C$ have a common point in $\P^n$. If the curve $C$ lies in $W_L$ for some $L\in \alpha^1(e)$, then the center of $X$ contains a point. We then project from such points until $C$ is not contained in $W_L$ for any $L\in \alpha^1(e)$. 
  Then the linear system $\fL_e$ restriced to $C$ has dimension $n-1$ by Proposition \ref{prop:cone}. We write $\iota\colon C \to \P^{n-1}$ for the map corresponding to this base-point free linear system. If $X$ is hyperbolic, then by Proposition \ref{prop:hyp} the map $\iota$ has the property that every hyperplane intersects $\iota(C)$ only in real points. By Proposition~\ref{prop:noline} this cannot be since $\iota(C)$ is not a line now that it is not contained in $W_L$ for any $L\in \alpha^1(e)$.
  
  Now assume that $X\subset\pp^n$ is not the cone over a plane curve. Then the center $\cent(X)$ of $X$ has dimension at most $n-4$. Consider a three-plane $E\subset\pp^n$ that contains $e$ and that does not intersect $\cent(X)$.
  Then $X' = X\cap E\subset\pp^3$ is a hypersurface ruled by a one-dimensional family $C'\subset \Gr(1,3)$ of lines in $\P^3$ obtained from $C$ by the rational map $\Lambda \mapsto \Lambda \cap E$ from $\Gr(n-2,n)$ to $\Gr(1,3)$. But it has at least one point in its center by the first part of the proof, which is a contradiction to Proposition~\ref{prop:center}.

  The last part follows from the Helton--Vinnikov Theorem \cite{HV07} which says that hyperbolicity cones of ternary hyperbolic polynomials, and thus cones thereover, are spectrahedral.
\end{proof}

\begin{ex}
  In \cite{convexhullcircle} the convex hull $K$ of the union of two plane circles in $\R^3$ (not both in the same plane) was studied. They show that this is a spectrahedron if and only if the Zariski closure $X\subset\pp^3$ of the two circles is a complete intersection \cite[Thm.~2.1]{convexhullcircle}. This result also follows from our \Cref{thm:onedimrule}. Indeed, the Zariski closure of $\partial K$ contains a surface $S\subset\pp^3$ that is ruled by lines that intersect each of the (Zariski closures of the) two circles. If $K$ is a spectrahedron, then $S$ must be hyperbolic and thus by \Cref{thm:onedimrule} it is the cone over a plane hyperbolic curve. Clearly, this plane curve can be chosen to be one of the two circles. Thus $S$ is a singular quadric and $X$ the complete intersection of $S$ with the union of the two planes spanned by the circles. The other direction is straight-forward.
\end{ex}

In general, we get the following result for convex hulls of curves in three-space.

\begin{cor}\label{cor:hypisspec}
 The closed convex hull of a one-dimensional semi-algebraic set in $\R^3$ is a spectrahedron if and only if it is an affine slice of a hyperbolicity cone.
\end{cor}

\begin{proof}
  Since a spectrahedron is always an affine slice of a hyperbolicity cone, one implication is trivial. For the other implication, let $S\subset \R^3$ be a one-dimensional semi-algebraic set. We now assume that the algebraic boundary of the closed convex hull $C$ of $S$ is a hyperbolic polynomial. It has irreducible components that come from the $2$-dimensional faces of $C$. Their Zariski closures are affine hyperplanes. The other irreducible components are the Zariski closure of $1$-dimensional families of $1$-dimensional faces and therefore hypersurfaces in $\A^3$ with a $1$-dimensional ruling. Since every factor of a hyperbolic polynomial is again hyperbolic, our \Cref{thm:onedimrule} implies that these hypersurfaces are spectrahedral. So $C$ is the intersection of spectrahedra and therefore a spectrahedron itself. 
\end{proof}

\begin{ex}\label{ex:ellconv3}
 Now let $C\subset\pp^3$ be an elliptic normal curve. Then $C$ is the complete intersection of two quadratic forms $p$ and $q$. The span of $p$ and $q$ over $\C$ contains exactly four singular irreducible quadrics $p_1,\ldots,p_4$. Now let $K$ be the convex hull of a very compact component $C_0$ of $C(\R)$. The boundary of $K$ is contained in the \emph{edge surface}, i.e. the union of all stationary bisecant lines, see \cite{raneconv}. These are lines that intersect $C$ in two points that share a tangent plane. Clearly, the lines from the ruling of a singular quadric that contain $C$ are stationary bisecants. On the other hand, the edge surface has degree eight by \cite[Thm.~2.1]{raneconv} and therefore it is the zero set of $q_1\cdots q_4$. The Zariski closure of the boundary of $K$ will then consist of those two singular quadrics $q_i$ and $q_j$ that contain (real) lines intersecting $C_0$ in an even number of points. These are hyperbolic and thus $K$ is a spectrahedron by \Cref{cor:hypisspec}. 
 \begin{figure}[h]
  \begin{center}
    \includegraphics[width=5cm]{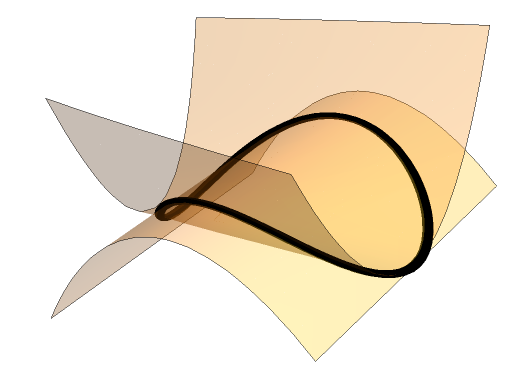}
  \end{center}
  \caption{A connected component of an elliptic normal curve and the algebraic boundary of its convex hull.}
  \label{fig:elliptichull}
  \end{figure}

\end{ex}

There are more interesting hyperbolic hypersurfaces ruled by some $R\subset \Gr(k,n)$ if we allow $\dim(R)>1$, as the following example shows.

\begin{ex}
  Here we describe hyperbolic hypersurfaces in $\pp^4$ that arise as the join of two plane hyperbolic curves $C_1$ and $C_2$, i.e. we have a ruling of lines isomorphic to $C_1\times C_2$. In $\pp^4$ we consider the two planes $H_1=\cV(x_1,x_2)$ and $H_2=\cV(x_3,x_4)$. In each $H_i$ we consider a plane curve that is hyperbolic with respect to $e=(1,0,0,0,0)$. We claim that the join $X=J(C_1,C_2)\subset\pp^4$ is a hypersurface hyperbolic with respect to $e$. Indeed, let $L\subset\pp^4$ be a line passing through $e$. We consider the set $V_L$ of pairs $(p,q)\in H_1\times H_2$ such that the span of $p$ and $q$ intersects $L$. One checks that $V_L=L_1\times L_2$ for two lines $L_i\subset H_i$, $i=1,2$, both containing $e$. Thus if $x\in L\cap X$, then $x$ is in the span of two points $p_i\in L_i\cap C_i$, $i=1,2$, which are both real because $C_i$ is hyperbolic. Thus $x$ is real.
\end{ex}

We will see more such examples in the following section where we discuss secant varieties of curves.

\section{Hyperbolic secant varieties of $M$-curves}\label{sec:hypsecMcurve}
The following construction due to Gabard \cite{gabard} will be our guiding example for constructing curves whose secant varieties are hyperbolic.
\begin{ex}\label{ex:mcurvemap}
Let $X$ be a smooth $M$-curve of genus $g$ and let $X(\R) = S_0\cup S_1 \cup \ldots \cup S_g$ be a decomposition into connected components so that $S_i$ is topologically isomorphic to $S^1$. Consider the divisor $$D=P_1+\ldots+P_n+Q_1+\ldots+Q_g,$$ where $P_i$ are distinct points on $S_0$ and the point $Q_j$ lies on $S_j$ for every $j\geq 1$. By Riemann's inequality, we have
\[
\ell(D) \geq n+1.
\]
In particular, the corresponding line bundle $\sL(D)$ has a global section that vanishes at any choice of $n$ points on $S_0$.
Choose two global sections $a_0$ and $a_1$ of $\sL(D)$, which vanish at $P_1^{(0)},P_2^{(0)},\ldots,P_n^{(0)}$ and $P_1^{(1)},P_2^{(1)},\ldots,P_n^{(1)}$, respectively, with the property that they alternate along the circle $S_0 = S^1$, i.e.~the order in clockwise (or counterclockwise) direction is $P_1^{(0)},P_1^{(1)},P_2^{(0)},P_2^{(1)},\ldots$.
Since they have $n+g$ zeros and the parity of zeros on any given connected component is constant under linear equivalence (see \cite[Lemma~4.1]{grossharris}), they must have one zero on each of the connected components $S_1,S_2,\ldots,S_g$ each. In particular, their zeros interlace on every connected component of $X$.
Now extend the chosen sections to a basis $(a_0,a_1,a_2,\ldots,a_m)$ of $H^0(X,\sL(D))$.
Then the image of $X$ under the map $\phi_{\sL(D)} \colon X \to \P(H^0(X,\sL(D))^\vee)$ is hyperbolic with respect to the subspace $E = \sV(x_0,x_1)\subset \P(H^0(X,\sL(D))^\vee)$.
Indeed, the projection away from $E$ is the map $\wh{\phi} \colon  X\to \P^1$, $x\mapsto (a_0(x):a_1(x))$, which is real fibered because
\[
\wh{\phi}^{-1}(\lambda:\mu) = \{x\in X \colon \mu a_1(x) - \lambda a_0(x) = 0\},
\]
which consists only of real points for $(\lambda:\mu)\in\P^1(\R)$ because the zeros of $a_0$ and $a_1$ interlace, see \cite[Lemma~2.10]{kummersep}.
\end{ex}

In the previous example we have constructed a linear system $\fD$ of dimension one on an $M$-curve $X$ that has the property that every $D\in\fD$ satisfies $\textrm{supp}(D)\subset X(\R)$. We will call such a linear system on a curve \emph{totally real}. Totally real linear systems can only exist on curves of dividing type. Conversely, every curve of dividing type admits a totally real linear system. This essentially goes back to Ahlfors \cite{Ahl50}. In the following we want to generalize this notion to linear systems of higher dimension. However, a linear system of dimension at least two cannot have the property that all its members have totally real support by \Cref{prop:noline}.

\begin{Def}\label{def:vastlyreal}
A linear system $\fD$ of dimension $2k+1$ on a curve $X$ is \emph{{\placeholderone}} if $|\textrm{supp}(D)\setminus X(\R)|\leq 2k$ for all $D\in\fD$.
\end{Def}

\begin{rem}
 A {\placeholderone} linear system on $X$ of dimension one, i.e. $k=0$, is a real fibered morphism $X\to\pp^1$. The case $k=1$ has been extensively studied by Mikhalkin and Orevkov in \cite{mikhalkin2019maximally} where they give numerous equivalent conditions to being {\placeholderone}. For example they characterize the topology of the corresponding embeddings to $\pp^3$: These are \emph{maximally writhed algebraic links}.
\end{rem}

In the following we will characterize {\placeholderone} linear systems in terms of hyperbolicity of secant varieties. Unlike in the case $k=0$, we will deduce from \cite{mikhalkin2019maximally} that, for $k=1$, {\placeholderone} linear systems can only exist on $M$-curves.

\begin{lem}\label{lem:generalsecant}
Let $k$ be a positive integer with $2k+1\leq n$ and let $X\subset \P^{n}$ be an irreducible and nondegenerate curve. A general point $p\in\Sec_k(X)$ lies only on $k$-secants $\Lambda$ to $X$ with $|\Lambda\cap X| = k+1$.
\end{lem}

\begin{proof}
The secant varieties of an irreducible and nondegenerate curve have the expected dimension $2k+1$ as long as $2k+1\leq n$, see \cite{lange}. So every general point $p\in\Sec_k(X)$ lies on finitely many $k$-secants to $X$.

Consider the rational map from the $(k+1)$-fold product of $X$ to $\Gr(k,2k+1)$ given by mapping a tuple of points $(x_0,x_1,\ldots,x_k)$ to the Pl\"ucker coordinates of the $k$-plane they span. This is a generically one-to-one map by the General Position Lemma \cite[Chapter III, \S1]{ACGHI}. The $k$-secants of $X$ that contain more than $k+1$ points are contained in the part of the image, where the map fails to be one to one. In particular, the set of special $k$-secants to $X$ has dimension at most $k$. Therefore, the union in $\P^{n}$ over all special $k$-secants to $X$ has dimension at most $2k$. So a general point does not lie on such a special $k$-secant.
\end{proof}

\begin{lem}\label{lem:realgeneral}
Let $X\subset \P^{n}$ be an irreducible nondegenerate real curve.
Let $k$ be a positive integer with $2k+1\leq n$. A general real point $p\in\Sec_k(X)(\R)$ either lies on a $k$-secant $\Lambda$ of $X$ with the property that $\Lambda\cap X = \ol{\Lambda}\cap X$ or on a $k$-secant $\Lambda$ of $X$ with the property $(\Lambda\cap X)\cap (\ol{\Lambda}\cap X) = \emptyset$.
\end{lem}

\begin{proof}
Let $p\in\Sec_k(X)(\R)$ be a general real point that does not lie on $\Sec_{k-1}(X)$ and not on a $k$-secant $\Lambda$ of $X$ with $\Lambda\cap X = \ol{\Lambda}\cap X$. By Lemma~\ref{lem:generalsecant}, the point $p$ lies on a $k$-secant $\Lambda$ of $X$ with $|\Lambda\cap X| = k+1$. Let $j$ be the cardinality of $(\Lambda\cap X)\cap (\ol{\Lambda}\cap X)$.

Suppose that $j>0$ and let $m$ be the dimension of $\Lambda\cap \ol{\Lambda}$. The intersection of $\Lambda \cap \ol{\Lambda}$ with $X$ consists of $j$ points whose span does not contain $p$. Thus $1\leq j\leq m$. This implies that the dimension of the span $U$ of $\Lambda\cup \ol{\Lambda}$ is at most $2k-m\leq 2k-j$. On the other hand, $U$ contains at least $2k-j+2$ points of $X$ which span $U$. This means that $U$ is spanned by $N=2k-m+1$ points on $X$ but intersects $X$ more than $N$ points. However, by the General Position Lemma \cite[Chapter III, \S1]{ACGHI} and since $N<n$, a subspace that is spanned by $N$ general points on $X$ intersects $X$ only in these points. Therefore, our subspace $U$ is special among those subspaces spanned by $N$ points on $X$ and thus the dimension of all such $U$ is at most $N-1=2k-m$.

The original point $p$ is in special position in $U$: it lies on the intersection of $\Lambda\cap \ol{\Lambda}$, i.e.~an $m$ dimensional space that arises as the intersection of two $k$-dimensional subspaces of $U$ that are each spanned by $k+1$ points in $U\cap X$. Since $|U\cap X|$ is finite, there are only finitely many such $k$-planes in $U$. The set of points that lie on such an $m$-plane therefore has dimension at most $(2k-m)+m=2k<\dim \Sec_k(X)$. Thus a general point will not lie on such an $m$-plane.
\end{proof}

\begin{thm}\label{thm:secantchara}
Let $X\subset \P^n$ be an irreducible nondegenerate real curve. Let $\Sec_k(X)$ be its $k$-th secant variety and suppose that $\Sec_k(X)\neq \P^n$. Let $E\subset \P^n$ be a real linear subspace of dimension $\codim(\Sec_k(X)) - 1$ such that $E\cap \Sec_k(X) = \emptyset$. The following are equivalent:
\begin{enumerate}
    \item The linear system on $X$ defined by all hyperplanes $H\subset \P^n$ containing $E$ is {\placeholderone}.
    \item $\Sec_k(X)$ is hyperbolic with respect to $E$.
\end{enumerate}
\end{thm}

\begin{proof}
  $(1)\Rightarrow(2):$
Suppose $\Sec_k(X)$ is not hyperbolic with respect to $E$. Then there exists a real linear subspace $E'\subset\P^n$ of dimension $\dim(E)+1$ such that $E\subset E'$ and $E'\cap \Sec_k(X)$ contains a non-real point. This is an open condition on $E'$.
Let $\pi_E\colon \P^n \ratto \P^{2k+1}$ be the linear projection with center $E$. Then $\pi_E(X)$ is an irreducible and nondegenerate real curve and $\Sec_k(\pi_E(X))$ is (up to closure) equal to $\pi_E(\Sec_k(X))$, which is Zariski dense in $\P^{2k+1}$ by \cite{lange}.

The point $p = \pi_E(E')\in\P^{2k+1}(\R)$ can be chosen to be general, because we can perturb $E'$.
Combining Lemma~\ref{lem:realgeneral} with the fact that $E'\cap \Sec_k(X)$ contains a non-real point, we conclude that $p$ lies on a $k$-secant $\Lambda$ of $\pi_E(X)$ with $\Lambda\cap X \cap \ol{\Lambda} = \emptyset$.

Consider the linear subspace of $\P^{2k+1}$ spanned by $p$, $\Lambda$, and $\ol{\Lambda}$. Its dimension is at most $2k$ because $p$ is contained in $\Lambda \cap \ol{\Lambda}$. Hence there is a hyperplane $H\subset \P^n$ containing $E$ and $\pi_E^{-1}(\Lambda \cup \ol{\Lambda})$. Since $\pi_E\vert_X\colon X\to \P^{2k+1}$ is a real morphism of curves, $H\cap X$ contains at least $2k$ non-real points, namely preimages of the points spanning $\Lambda$ and $\ol{\Lambda}$. 

$(2)\Rightarrow(1):$ Let $\Sec_k(X)$ be hyperbolic with respect to $E$ but assume for the sake of contradiction that there is a real hyperplane $H\subset \P^n$ containing $E$ that intersects $X$ in more than $2k$ non-real points. Since these non-real points come in complex conjugate pairs, there are at least $2k+2$ non-real points. Any sufficiently small perturbation of the pair $E\subset H$ will also have these properties. Thus by the General Position Lemma \cite[Chapter III, \S1]{ACGHI} we can assume that $H$ intersects $X$ in $2k+2$ non-real points $P_0,\ldots,P_k,\overline{P_0},\ldots,\overline{P_k}$ that are projectively independent.
This implies in particular that the $k$-secant $\Lambda$ spanned by $P_0,\ldots,P_k$ does not intersect its complex conjugate $\overline{\Lambda}$. Thus $\Lambda$ does not contain any real point. The images of $\Lambda$ and $\overline{\Lambda}$ under the linear projection from center $E$ will however intersect for dimension reasons as they are both contained in the $2k$-dimensional image of $H$. This intersection point is real and its preimage is a real linear space $E'$ of dimension $\codim(\Sec_k(X))$ that contains $E$ and intersects $\Lambda$ in a necessarily non-real point. This contradicts the assumption that $\Sec_k(X)$ is hyperbolic with respect to $E$.
\end{proof}

\begin{rem}\label{rem:nonm}
  Assume that in \Cref{thm:secantchara} the curve $X$ is smooth and $k>0$. Let us project $X$ to $\pp^{2k+1}$ from $E$. After perturbing $E$ if necessary the image $X'$ is isomorphic to $X$ and has the property, that any real hyperplane intersects $X'$ in at most $2k$ non-real points. Thus if $k=1$, then $X'\subset\pp^3$, and therefore $X$, is an $M$-curve by \cite[Thm.~2]{mikhalkin2019maximally}. We do not know whether there are {\placeholderone} linear systems of dimension $>3$ on curves that are not $M$-curves. At least if we assume that not every tuple of $2k-2$ points on $X$ lies on a $(2k-1)$-secant of dimension $2k-2$, the answer is negative. Indeed, we let $E'$ be the span of $k-1$ generic pairs of complex conjugate points on $X$ and project $X'$ from $E'$ to $\pp^3$. The image $X''$ is then again isomorphic to $X$ and we observe that every real plane in $\pp^3$ intersects $X''$ in at most two non-real points. Then by the same argument as above $X$ is an $M$-curve. 
\end{rem}

\begin{rem}
  There is yet another way to look at \Cref{thm:secantchara}. Let $X\subset\pp^{2k+1}$ be a curve embedded via a {\placeholderone} linear system $\fD$ of dimension $2k+1$. From this we obtain the map $X^{(k+1)}\to\G(k,n)$ from the $(k+1)$st symmetric power of $X$ to the Grassmannian of $k$-planes that sends a tuple of $k+1$ points on $X$ to the $k$-secant they span. Such a map defines a vector bundle $\cF$ of rank $k+1$ on $X^{(k+1)}$. Moreover, there is a subspace $V\subset H^0(X^{(k+1)},\cF)$ of dimension $2k+2$ such that we can identify $\pp(V)$ with $\pp^{2k+1}$ in a way that the zero set of any $s\in V$ corresponds to the $k$-secants that contain the corresponding point in $\pp^{2k+1}$. A generic real point in $\pp^{2k+1}$ lies only on real $k$-secants of $X$. Thus the space of sections $V$ is totally real in the sense that a generic section in $V$ has only real zeros.

As a concrete example let $X\subset\pp^3$ be the projection of the rational normal curve of degree four from a point $p\in\pp^4$ that does not lie on the secant variety $\Sec_1(X)$. The vector bundle $\cF$ on $X^{(2)}=\pp^2$ that we get is just the tangent bundle $\cT_{\pp^2}$. In order to describe the vector space $V$, we realize $\cT_{\pp^2}$ as the cokernel of $$\begin{pmatrix}{p_0x_0+p_1x_1+p_2x_2}\\{p_1x_0+p_2x_1+p_3x_2}\\{p_2x_0+p_3x_1+p_4x_2}\end{pmatrix}.$$Then $V$ is spanned by the columns of the matrix $$\begin{pmatrix}
  x_0&x_1&x_2&0&0\\0&x_0&x_1&x_2&0\\0&0&x_0&x_1&x_2
\end{pmatrix}.$$ If $\Sec_1(X)$ is hyperbolic with respect to $p$, then a generic section $s\in V$ has only real zeros on $\pp^2$.
\end{rem}

We now show that on any $M$-curve {\placeholderone} linear systems of any dimension exist. More precisely, any complete linear system corresponding to a divisor, that is {\placeholdertwo} in the following sense, contains a {\placeholderone} linear system.

\begin{Def}
  Let $X$ be an $M$-curve with $S_0,\ldots,S_g$ being the connected components of $X(\R)$. A divisor $D$ on $X$ is \emph{{\placeholdertwo}} if it has odd degree on $S_1,\ldots,S_g$.
\end{Def}

\begin{lem}\label{lem:tollnonspecial}
  If $D$ is {\placeholdertwo}, then $D$ is nonspecial, i.e.  $\ell(K-D)=0$.
\end{lem}

\begin{proof}
  This is a direct consequence of \cite[Thm. 2.5]{HuismanMR1969741}.
\end{proof}

\begin{lem}\label{rem:verycompact}
  Let $D$ be very ample, {\placeholdertwo} and assume that it has even degree on $S_0$. Then $S_0$ is very compact. 
\end{lem}

\begin{proof}  
By assumption we have $\deg(D) = 2k + g$. Then there is a hyperplane $H\subset \P^n$ such that the support of $H.C$ contains $k$ pairs of complex conjugate non-real points. Since the support of $H.C$ on the connected components $S_i$ ($i>0$) has at least $1$ point, $H$ does not intersect $S_0$.
\end{proof}

\begin{thm}\label{thm:tollimpliesveryreal}
  Let $X$ be an $M$-curve of genus $g$. For any {\placeholdertwo} divisor $D$ of degree $d\geq 2k+g+1$ there exists a {\placeholderone} linear system $\fD\subset|D|$ of dimension $2k+1$.
\end{thm}

\begin{proof}
 Let $n=d-k-g$ and choose points $$P_1^{(0)},P_2^{(0)},\ldots,P_n^{(0)}\textrm{ and }P_1^{(1)},P_2^{(1)},\ldots,P_n^{(1)}$$ on $S_0$ as in \Cref{ex:mcurvemap}, i.e. they alternate along the circle $S_0 $ meaning that the order in clockwise (or counterclockwise) direction is $P_1^{(0)},P_1^{(1)},P_2^{(0)},P_2^{(1)},\ldots$. We let $V_i\subset\Gamma(X,\sL(D))$ be the space of all sections that vanish on $P_1^{(i)},P_2^{(i)},\ldots,P_n^{(i)}$ for $i=0,1$. Since $$D-\sum_{j=1}^nP_j^{(i)}$$ is {\placeholdertwo}, the dimension of $V_i$ is $k+1$ by Lemma \ref{lem:tollnonspecial}. For the same reason we have $V_0\cap V_1=\{0\}$. Thus $V=V_0+V_1$ has dimension $2k+2$ and we are going to prove that the linear system corresponding to $V$ is {\placeholderone}. Let $s_i\in V_i$ and $Z_i\subset X$ its zero set. We need to show that $s=s_0+s_1$ has at most $2k$ non-real zeros. By construction there are $n$ connected components of $S_0\setminus Z_0$ that contain one of the $P_j^{(1)}$. Among the zeros of $s_1$ are the points $P_j^{(1)}$ and, since $D$ is {\placeholdertwo}, $s_1$ has at least one zero on each $S_k$ for $k\geq 1$. Thus there are only $d-n-g=k$ additional zeros of $s_1$ that might lie on $S_0$. We conclude that there are at least $n-k$ connected components of $S_0\setminus Z_0$ that contain an odd number of zeros of $s_1$. Thus $s$ has at least $n-k$ zeros on $Z_0$ by the intermediate value theorem. In addition to that there is at least one zero of $s$ on each of the $S_k$ for $k\geq 1$. Therefore, there are in total at least $n-k+g=d-2k$ real zeros of $s$.
\end{proof}

\begin{rem}\label{rem:degreepartitions}
  In the case $k=0$ resp. $k=1$ it was shown in \cite{kummersep} resp. \cite{mikhalkin2019maximally} that for any tuple $(d_0,\ldots,d_g)$ there is an $M$-curve $X$ and a {\placeholderone} linear system $\fD$ of dimension $2k+1$ and degree $2k+\sum_{j=0}^gd_0$ such that any $E\in\fD$ has at least $d_i$ zeros on $S_i$ for each $i$. It would be very interesting to determine whether the corresponding statement remains true in the case $k>1$. The tuples realized by our construction are all of the type $(d_0,1,\ldots,1)$ for some $d_0\geq1$. In particular, for $k\leq1$ not every {\placeholderone} linear system arises from our construction. Another compelling question is whether the rigid isotopy results from \cite{mikhalkin2019rigid} generalize to the cases $k>1$.
\end{rem}

\begin{cor}\label{cor:tollhyp}
  Let $X\subset \pp^n$ be an $M$-curve embedded via the complete linear system of a {\placeholdertwo} divisor and $n\geq 2k+1$. The secant variety $\sigma_k(X)$ is hyperbolic.
\end{cor}

\begin{proof}
  This is just the combination of Theorems \ref{thm:secantchara} and \ref{thm:tollimpliesveryreal}.
\end{proof}

\begin{ex}
  By definition any divisor on $\pp^1$ is {\placeholdertwo}. Therefore, any secant variety $\sigma_k(C)$ of a rational normal curve $C\in\pp^n$ is hyperbolic.
\end{ex}

\begin{ex}\label{ex:elliptic1}
  First note that any divisor of odd degree on an elliptic $M$-curve is {\placeholdertwo}. Thus it is very easy to construct hyperbolic secant varieties of elliptic curves. Here we give an explicit example.
  Let $C\subset \P^2$ be the smooth cubic curve defined by the equation $x_0^3-x_0x_2^2-x_1^2x_2$. We embed $C$ to $\P^4$ by the map
  \[
    \iota\colon \left\{
    \begin{array}{l}
      C \ratto \P^4 \\
      (x_0:x_1:x_2) \mapsto (x_0^2:x_0x_1:x_0x_2:x_1x_2:x_2^2)
    \end{array}\right.
  \]
  given by the linear system of quadrics on $\P^2$ vanishing at $(0:1:0)$, which are global sections of $\sL(5P)$ with $P = (0:1:0)$. The equation of its secant variety in coordinates $(z_0:z_1:z_2:z_3:z_4)$ on $\P^4$ is 
  \[
    \begin{array}{cl}
    f =  z_0^3z_2^2+z_1^2z_2^3-z_0z_2^4-4z_0z_1z_2^2z_3+z_0^2z_2z_3^2+z_2^3z_3^2+2z_1z_2z_3^3-z_0z_3^4-z_0^4z_4+\\
     z_2^4z_4+2z_0^2z_1z_3z_4+2z_1z_2^2z_3z_4-z_1^2z_3^2z_4-2z_0z_2z_3^2z_4+z_0^3z_4^2-z_1^2z_2z_4^2-z_0z_2^2z_4^2.
    \end{array}
  \] This is a hyperbolic polynomial by \Cref{cor:tollhyp}. A direction of hyperbolicity is given by $(2:0:-3:0:6)$. This hypersurface is singular exactly at the curve $C$.
\end{ex}

\begin{ex}
  We want to construct a similar example in genus two. Consider the hyperelliptic $M$-curve $X$ defined by $$y^2=(1-x^2)\cdot(4-x^2)\cdot(9-x^2).$$ The following points on $X$ define a {\placeholdertwo} divisor $D$ of degree six: $$(\pm 2,0),(\pm 1,0),(0,\pm 6).$$The corresponding embedding to $\pp^4$ is given by
  $$X\to\pp^4,\,(x,y)\mapsto(y:x^2-9:x^3-9x:xy:x^4-81). $$The equation of its secant variety in coordinates $(z_0:z_1:z_2:z_3:z_4)$ on $\P^4$ is
  \[\tiny{
  \begin{array}{cl}
    f=&13122 z_0^6 z_1^2+129762 z_0^4 z_1^4-661122 z_0^2 z_1^6-
    6936930 z_1^8+14418 z_0^4 z_1^2 z_2^2-451116 z_0^2 z_1^4 z_2^2+
    560690 z_1^6 z_2^2 \\ & -8162 z_0^2 z_1^2 z_2^4+137410 z_1^4 z_2^4-1170 
    z_1^2 z_2^6+2916 z_0^5 z_1 z_2 z_3+134136 z_0^3 z_1^3 z_2 z_3-
    582156 z_0 z_1^5 z_2 z_3 \\ & +5544 z_0^3 z_1 z_2^3 z_3-129368 z_0 z_1
    ^3 z_2^3 z_3+324 z_0 z_1 z_2^5 z_3-4374 z_0^4 z_1^2 z_3^2+34020 
    z_0^2 z_1^4 z_3^2-286902 z_1^6 z_3^2 \\ & +23940 z_0^2 z_1^2 z_2^2 z_3^
    2-29956 z_1^4 z_2^2 z_3^2+160 z_0^2 z_2^4 z_3^2-3542 z_1^2 z_2^4 
    z_3^2-648 z_0^3 z_1 z_2 z_3^3-3816 z_0 z_1^3 z_2 z_3^3 \\ & +616 z_0 
    z_1 z_2^3 z_3^3+486 z_0^2 z_1^2 z_3^4-3942 z_1^4 z_3^4-438 z_1^2 
    z_2^2 z_3^4+36 z_0 z_1 z_2 z_3^5-18 z_1^2 z_3^6-2187 z_0^6 z_1 
    z_4 \\ & -62613 z_0^4 z_1^3 z_4+237015 z_0^2 z_1^5 z_4+5241313 z_1^7 z_4-3015 z_0^4 z_1 z_2^2 z_4+164006 z_0^2 z_1^3 z_2^2 z_4 \\ & -134819 z_1^5 z_2^2 z_4+1459 z_0^2 z_1 z_2^4 z_4-50701 z_1^3 z_2^4 z_4+207 
    z_1 z_2^6 z_4-162 z_0^5 z_2 z_3 z_4 \\ & -38700 z_0^3 z_1^2 z_2 z_3 
    z_4+201750 z_0 z_1^4 z_2 z_3 z_4-468 z_0^3 z_2^3 z_3 z_4+34908 
    z_0 z_1^2 z_2^3 z_3 z_4-18 z_0 z_2^5 z_3 z_4 \\ & +567 z_0^4 z_1 z_3
    ^2 z_4-11718 z_0^2 z_1^3 z_3^2 z_4+149843 z_1^5 z_3^2 z_4-4354 z_0^2 z_1 z_2^2 z_3^2 z_4+12030 z_1^3 z_2^2 z_3^2 z_4 \\ & +607 z_1 z_2^
    4 z_3^2 z_4+36 z_0^3 z_2 z_3^3 z_4+52 z_0 z_1^2 z_2 z_3^3 z_4-
    52 z_0 z_2^3 z_3^3 z_4-45 z_0^2 z_1 z_3^4 z_4+1139 z_1^3 z_3^4 
    z_4 \\ & +49 z_1 z_2^2 z_3^4 z_4-2 z_0 z_2 z_3^5 z_4+z_1 z_3^6 z_4+
    81 z_0^6 z_4^2+10458 z_0^4 z_1^2 z_4^2-16391 z_0^2 z_1^4 z_4^2-
    1683152 z_1^6 z_4^2 \\ & +153 z_0^4 z_2^2 z_4^2-22028 z_0^2 z_1^2 z_2^2 
    z_4^2-12265 z_1^4 z_2^2 z_4^2-65 z_0^2 z_2^4 z_4^2+7026 z_1^2 z_2
    ^4 z_4^2-9 z_2^6 z_4^2 \\ & +3600 z_0^3 z_1 z_2 z_3 z_4^2-23272 z_0 z_1^3 z_2 z_3 z_4^2-3120 z_0 z_1 z_2^3 z_3 z_4^2-18 z_0^4 z_3^2 z_4^2+1538 z_0^2 z_1^2 z_3^2 z_4^2 \\ & -30776 z_1^4 z_3^2 z_4^2+188 z_0^
    2 z_2^2 z_3^2 z_4^2-1682 z_1^2 z_2^2 z_3^2 z_4^2-26 z_2^4 z_3^2 
    z_4^2+72 z_0 z_1 z_2 z_3^3 z_4^2+z_0^2 z_3^4 z_4^2 \\ & -104 z_1^2 z_3^4 z_4^2-z_2^2 z_3^4 z_4^2-729 z_0^4 z_1 z_4^3-3690 z_0^2 z_1^3 
    z_4^3+297839 z_1^5 z_4^3+1294 z_0^2 z_1 z_2^2 z_4^3 \\ & +7154 z_1^3 z_2^2 z_4^3-433 z_1 z_2^4 z_4^3-108 z_0^3 z_2 z_3 z_4^3+732 z_0 z_1^2 z_2 z_3 z_4^3+92 z_0 z_2^3 z_3 z_4^3 \\ & -90 z_0^2 z_1 z_3^2 z_4
    ^3+3106 z_1^3 z_3^2 z_4^3+98 z_1 z_2^2 z_3^2 z_4^3-4 z_0 z_2 z_3
    ^3 z_4^3+3 z_1 z_3^4 z_4^3+18 z_0^4 z_4^4+712 z_0^2 z_1^2 z_4^4 \\ & -
    31370 z_1^4 z_4^4-28 z_0^2 z_2^2 z_4^4-904 z_1^2 z_2^2 z_4^4+10 z_2^4 z_4^4+36 z_0 z_1 z_2 z_3 z_4^4+2 z_0^2 z_3^2 z_4^4-154 z_1^
    2 z_3^2 z_4^4 \\ & -2 z_2^2 z_3^2 z_4^4-45 z_0^2 z_1 z_4^5+1967 z_1^3 
    z_4^5+49 z_1 z_2^2 z_4^5-2 z_0 z_2 z_3 z_4^5+3 z_1 z_3^2 z_4^5+
    z_0^2 z_4^6-68 z_1^2 z_4^6-z_2^2 z_4^6 \\ & +z_1 z_4^7
  \end{array}}
  \]In order to find a point in the hyperbolicity cone of this polynomial, let $$P_1^{(0)}=(-1,0),P_2^{(0)}=(0,6),P_3^{(0)}=(0,-6),$$ $$P_1^{(1)}=(-\frac{1}{2},\frac{15}{8}\sqrt{7}),P_2^{(1)}=(1,0),P_3^{(1)}=(-\frac{1}{2},-\frac{15}{8}\sqrt{7}).$$Then with the notation as in the proof of \Cref{thm:tollimpliesveryreal} the vector space $V$ consists of all linear forms in the $z_i$ coordinates that vanish at the point $e=(0:3:-1:0:28)$. Therefore, the polynomial $f$ is hyperbolic with respect to $e$. Unlike in the case of elliptic curves, the singular locus of the secant variety does not only consist of the curve itself. Indeed, for every two points $P,Q$ on $X$ there are points $P',Q'$ such that $D-(P+Q+P'+Q')$ is a canonical divisor on $X$. Riemann--Roch implies that the linear space spanned by $P+Q+P'+Q'$ has dimension two. This shows that the secant spanned by $P$ and $Q$ intersects the one spanned by $P'$ and $Q'$. This point of intersection is a singular point of our secant variety. Since we have (at least) one such point for each secant, the singular locus of the secant variety is two-dimensional. Note however that these types of singularities do not occur on the boundary of the hyperbolicity cone of $f$. Finally, we note that there is a unique secant which is entirely contained in the singular locus: The one spanned by the points $R,S$ with the property that $R+S$ is linearly equivalent to $D-2K$. In our explicit example, these are the points $(1:0:0:3:0)\textrm{ and }(1:0:0:-3:0).$
\end{ex}

When $\sigma_k(X)$ is a hypersurface, we want to determine its hyperbolicity cone. 

\begin{lem}\label{prop:Dveryample}
Let $X$ be an $M$-curve and let $D$ be a {\placeholdertwo} divisor on $X$ of degree $n+g$ with $n\geq 3$.
\begin{enumerate}
\item The divisor $D$ is very ample with $\ell(D) = n+1$.
\item Any $m$ points on $\phi(S_0)\subset\P^n$ are in general linear position. Here $\phi:X\to\pp^n$ is the embedding corresponding to $D$.
\end{enumerate}
\end{lem}

\begin{proof}
We first show that $\ell(D) = n+1$. By Lemma \ref{lem:tollnonspecial} $D$ is non-special. Therefore, $\ell(D) = \deg(D) + 1 - g = n+1$. To show that $D$ is very ample, we use again \cite[Thm. 2.5]{HuismanMR1969741} to verify that $\ell(D-P) = \ell(D) - \deg P$ and $\ell(D-P-Q) = \ell(D)-\deg (P + Q)$ hold for all $P,Q \in X$. Note that $\deg P=2$ for any non-real point $P\in X$.

To show that any $m$ points on $\phi(S_0)$ are in general linear position, we use the same argument. First, we can assume that $m\leq n+1 = \deg(D) + 1 - g$. Pick $P_1,P_2,\ldots,P_m\in \phi(S_0)$. The divisor $D_{-m} = D-(P_1+P_2+\ldots + P_m)$ has odd degree on $S_1,\ldots,S_g$ and $\deg(D_{-m}) + g = n-m+g + g \geq 2g-1$ so that \cite[Theorem~2.5]{HuismanMR1969741} again shows that the divisor $D_{-m}$ is non-special. Therefore, we have
\[
\ell\left(D-(P_1+P_2+\ldots+P_m)\right)) = n+g-m + 1 - g = n + 1 - m,
\]
which shows that the span of the points $P_1,\ldots,P_m$ has the correct dimension $m-1$.
\end{proof}

\begin{prop}\label{prop:hypconetoll}
  Let $X\subset \pp^{2k}$ be an $M$-curve embedded via the complete linear system of a {\placeholdertwo} divisor. The secant variety $\sigma_{k-1}(X)$ is hyperbolic and its hyperbolicity cone is $\conv(S_0)$, where $S_0$ is the very compact connected component of $X(\R)$, see \Cref{rem:verycompact}.
\end{prop}

\begin{proof}
Let $f\in\R[x_0,x_1,\ldots,x_{2k}]$ be the hyperbolic polynomial defining $\Sec_{k-1}(X)$ in $\R^{2k+1}$. We will prove that the (closed) hyperbolicity cone $C$ of $f$ with respect to a point $e$ of hyperbolicity is the cone generated by $\widehat{S_0}\cap H_+$. Here $H_+$ is the half-space $\{x\in\R^{2k+1}\colon \langle x,e\rangle \geq 0\}$.

To show the inclusion $\cone(\widehat{S_0}\cap H_+)\subset C$, we use convex duality to show $\cone(\widehat{S_0}\cap H_+)^\vee \supset C^\vee$. By \cite[Example~3.15]{Sinn}, the convex cone $C^\vee$ is up to closure the convex cone generated by the regular real points on the dual variety $\Sec_{k-1}(X)^*_{\rm reg}(\R)\cap H_+$. A general point in $\ell\in\Sec_{k-1}(X)^*_{\rm reg}$ is tangent to $\Sec_{k-1}(X)$ at a general point and by Terracini's Lemma \cite[Proposition~4.3.2]{JoinsAndIntersections}, this means that $\ell$ vanishes on the tangent lines $X$ at $k$ regular points $p_1,p_2,\ldots,p_k\in X$. The linear form $\ell$ corresponds to the divisor $2(p_1+p_2+\ldots+p_k) + \sum_{i=1}^g q_i$ on $X$ for suitable points $q_i\in S_i$. It follows that $\ell$ has constant sign on $\widehat{S_0}\cap H_+$, which shows the desired inclusion.

To show the remaining inclusion $C\subset \cone(\widehat{S_0}\cap H_+)$, we argue by contradiction. Write $C'$ for $\cone(\widehat{S_0}\cap H_+)$ and suppose $C'$ were properly contained in $C$. Then there is a supporting hyperplane $\{x\in\R^{2k+1}\colon \ell(x)\geq 0\}$ to $C'$ separating a point $x\in C\setminus C'$ from $C'$. Let $F$ be the face of $C'$ supported by $\ell$.
The divisor on $X$ corresponding to $\ell$ has at least one point on $S_i$ for every $i>0$. Therefore, the face $F$ can have at most $k$ extreme points. This shows that the face $F$ contains a point of $\Sec_{k-1}(X)$, which contradicts the fact that hyperbolic polynomials are always non-zero on their hyperbolicity cones.
\end{proof}

\begin{cor}\label{cor:simplicial}
  Let $X\subset \pp^{2k}$ be an $M$-curve embedded via the complete linear system of a {\placeholdertwo} divisor. The hyperbolicity cone of $\Sec_{k-1}(X)$ is simplicial.
\end{cor}

\begin{proof}
By \Cref{prop:Dveryample}, any $k$ points on $S_0$ are in general linear position. Since the hyperbolicity cone of $\Sec_{k-1}(X)$ is the convex hull of $S_0$, this implies the claim.
\end{proof}

\begin{cor}
  Let $X\subset \pp^{2k}$ be an $M$-curve embedded via the complete linear system of a {\placeholdertwo} divisor. The hyperbolicity cone of $\Sec_{k-1}(X)$ has no $\Sym_+(\R^k)$-lift.
\end{cor}

\begin{proof}
  This follows from \Cref{cor:simplicial} and \cite[Thm.~1.4]{saunderson2019limitations} or \cite{averkov2019optimal}.
\end{proof}

\begin{ex}\label{ex:hankel}
  If $X\subset\pp^{2k}$ is a rational normal curve, then it is well known \cite[Prop.~9.7]{Ha95} that (after a suitable choice of coordinates) the hypersurface $\sigma_{k-1}(X)$ is defined by the determinant of the $(k+1)\times(k+1)$ Hankel matrix
\[H_k(x)=\begin{pmatrix}
  x_{0} & x_{1} & x_{2} & \ldots & \ldots  &x_{k}  \\
  x_{1} & x_2 &  &  & &\vdots \\
  x_{2} &  &  & & & \vdots \\ 
 \vdots & & & &  & x_{2k-2}\\
 \vdots & &  & & x_{2k-2}&  x_{2k-1} \\
x_{k} &  \ldots & \ldots & x_{2k-2} & x_{2k-1} & x_{2k}
\end{pmatrix}.
    \]The hyperbolicity cone of $\sigma_{k-1}(X)$ is the set of all $[x]$ where $H_k(x)$ is semidefinite. In \Cref{sec:elliptic} we will prove a similar result for elliptic normal curves.
\end{ex}

We end this section with an explicit example of a hyperbolic secant variety of a curve that is not embedded via a {\placeholdertwo} divisor. It would be useful to have a description of hyperbolicity cones in these situations as well.

\begin{ex}
 Consider the plane elliptic curve from \cite[Exp.~5.2]{kummersep} defined by $$-z^3+2xz^2-x^3+y^2z=0$$ and its image $X\subset\pp^5$ under the second Veronese map. It can be read off the link diagram in \cite[Fig.~4]{kummersep} that $X$ admits a linear projection to $\pp^3$ as a maximally writhed algebraic link. Thus $\sigma_1(X)$ is hyperbolic with respect to the center of this linear projection. Since any divisor obtained from a hyperplane section has even degree on each connected component of $X(\R)$, it cannot be {\placeholdertwo}. In the terminology of \Cref{rem:degreepartitions} this {\placeholderone} linear system realizes the pair $(2,2)$.
\end{ex}

\subsection{Real enumerative geometry}
Let $C\subset\pp^{2k}$ be a general nondegenerate curve of degree $d$ and genus $g$. Recall from \cite[p.~351]{ACGHI} that the number of secants spanned by $k+1$ points of $C$ that have dimension $k-1$ equals
\[
  \sum_{j=0}^{k+1}(-1)^{j}\binom{g+2k-d}{j}\binom{g}{k+1-j}.
\]
In this formula, $\binom{n}{k}$ may have a negative number $n$, in which case this is to be interpreted as $n(n-1)\cdot\ldots\cdot (n-k+1)/k!$. This means that $\binom{n}{k} = (-1)^k\binom{-n-1+k}{k}$ for $n<0$. The binomial coefficient is $0$, if $k>n\geq 0$. 

We want to show that for all $d,g,k$ with $d \geq 2k + g + 1$ there are real curves $C$ such that all of these special secants are real. Indeed, let $C$ be an $M$-curve of genus $g$ and let $D$ be a {\placeholdertwo} divisor of degree $d\geq 2k+g+1$. By \Cref{thm:tollimpliesveryreal} there exists a {\placeholderone} linear system $\fD\subset|D|$ of dimension $2k+1$. Thus by \Cref{thm:secantchara} the secant variety $\sigma_k(C)$ of $C\subset\pp^{d-g}$ embedded via the complete linear system $|D|$ is hyperbolic with respect to some linear space $E\subset\pp^{d-g}$. A general subspace $E'\supset E$ of codimension $2k+1$ intersects $\sigma_k(C)$ only in real $k$-secants. If we project from $E'$ to $\pp^{2k}$, exactly those $k$-secants of $C$ that intersect $E'$ will be mapped to special secants of the image of $C$ in $\pp^{2k}$ that have dimension $k-1$. Thus the above enumerative problem admits totally real instances.

\begin{thm}\label{thm:totallyreal}
  For all positive integers $d,g,k$ with $d \geq 2k + g + 1$ there is a real curve $C\subset\pp^{2k}$ such that all secants spanned by $k+1$ points of $C$ that have dimension $k-1$ are real.\qed
\end{thm}

For later reference we also record the degree of the secant varieties which follows from the above formula and the discussion afterwards.

\begin{lem}\label{lem:secantdeg}
  Let $C\subset\pp^{n}$ with $n>2k+1$ be a nondegenerate curve of degree $d$ and genus $g$. The variety $\sigma_k(C)$ has degree $D=\sum_{j=0}^{k+1}(-1)^{j}\binom{g+2k-d}{j}\binom{g}{k+1-j}$.\qed
\end{lem}

The cases we will later need are the following.

\begin{cor}\label{cor:deg}
  Consider a general, nondegenerate, projectively normal curve $C\subset\pp^{2k+2}$ of genus $g$. Then the degree of the hypersurface $\sigma_k(C)$ is equal to $$(k+2)\cdot\sum_{i=0}^{k+1}\binom{g}{i}-g\cdot\sum_{i=0}^{k}\binom{g-1}{i}.$$If $g\leq k+1$, then this number equals $(2k+4-g)\cdot 2^{g-1}$.
\end{cor}

\begin{proof} 
  Here we have $d=2k+g+2$ and the formula from Lemma \ref{lem:secantdeg} simplifies to $$\sum_{j=0}^{k+1}(-1)^j\binom{-2}{j}\binom{g}{k+1-j}=\sum_{j=0}^{k+1}(j+1)\binom{g}{k+1-j}=\sum_{i=0}^{k+1}(k+2-i)\binom{g}{i}.$$
  This further simplifies to $$(k+2)\cdot\sum_{i=0}^{k+1}\binom{g}{i}-\sum_{i=0}^{k+1}i\cdot\binom{g}{i}=(k+2)\cdot\sum_{i=0}^{k+1}\binom{g}{i}-g\cdot\sum_{i=0}^{k}\binom{g-1}{i}.$$
  Since $\binom{g}{i}=0$ for $i>g$ in the case $g\leq k+1$ this further simplifies to $$(k+2)\cdot\sum_{i=0}^{g}\binom{g}{i}-g\cdot\sum_{i=0}^{g-1}\binom{g-1}{i}=(k+2)\cdot 2^g-g\cdot 2^{g+1}.\qedhere$$
\end{proof}

\section{Elliptic normal curves}\label{sec:elliptic}
Throughout this section, $C$ is an irreducible smooth complete curve of genus $1$ with a smooth real point. Then any divisor of odd degree $n\geq 3$ on $C$ is {\placeholdertwo} and embeds the curve $C$ into $\P^{n-1}$. If $C$ is an $M$-curve, then the secant varieties $\sigma_j(C)$ embedded into $\P^{n-1}$ by a complete linear system are hyperbolic varieties by \Cref{cor:tollhyp}. In this section, we discuss this case in detail and in particular give the construction of a definite determinantal representation for the case that $\sigma_j(C)$ is a hypersurface in $\P^{n-1}$. 
For instance, given a line bundle $\sL$ of degree $5$ on $C$, we get an embedding $C\to \P(H^0(C,\sL)^\vee) \cong \P^4$. The secant variety of this elliptic normal curve is a hypersurface of degree $5$ in $\P^4$ that admits symmetric $5\times 5$ determinantal representations. The case of degree $5$ involves some nice geometry and is more accessible. It serves as a guiding example throughout this section. 

Secant varieties of elliptic normal curves $C\subset \P^n$ have been extensively studied, e.g. in \cite{hulekthedogsout,tommyfisher}. We rely in particular on results from the work \cite{tommyfisher}.
\begin{facts}\label{facts:ellnormal}
  Fisher shows that the homogeneous vanishing ideal of $\sigma_j(C)$ is generated by $\beta(j+2,n+1)$ forms of degree $j+2$ whenever $\codim(\sigma_j(C))\geq 2$, where $\beta(r,m) = \binom{m-r}{r} + \binom{m-r-1}{r-1}$ (see \cite[Theorem~1.2]{tommyfisher}). When $\codim(\sigma_j(C)) = 1$, then $\sigma_j(C)$ is a hypersurface of degree $2j+3$, see \Cref{cor:deg} or \cite[Theorem~1.1]{tommyfisher}.
  Furthermore, it is a consequence of \cite[Theorem~1.4]{tommyfisher} that the singular locus of $\sigma_j(C)$ is $\sigma_{j-1}(C)$ (see also \cite[Proposition~8.15]{hulekthedogsout}).
\end{facts}

We follow a well-established general strategy for constructing a symmetric determinantal representation for plane curves, which is usually referred to as Dixon's method and goes back to Hesse's paper \cite{HesseMR1579060}, see \cite{CAG, hvelemt}. The construction roughly goes as follows:

Let $f\in\R[x_0,x_1,\ldots,x_n]$ be a homogeneous polynomial of degree $d$ and suppose that $f = \det(A(x))$, where $A(x)$ is a symmetric $d\times d$ matrix whose entries are real linear forms. The construction goes after the adjugate matrix $\adj(A)$, which is a matrix that has rank $1$ at a general point of the hypersurface $\sV(f)$ because $A\cdot \adj(A) = \det(A) I_d = f(x) I_d$. The first row of this matrix $\adj(A)$ depends on the choice of a \emph{contact interlacer}, which is characterized for us in \Cref{lem:contactint} below. By symmetry, this also determines the first column of the matrix. Since the matrix is supposed to have rank $1$ modulo $f$, we can use the vanishing of the $2\times 2$-minors to complete this matrix and thereby determine $A$. 

Admitting a contact hypersurface is the central property used in classical algebraic geometry to construct symmetric determinantal representations. The real interlacing property implies the definiteness of the determinantal representation. 

We are not quite able to apply this construction directly to the secant varieties $\sigma_j(C)$ but rather take a detour via the symmetric product $C^{(j+1)}$, which is the product $C^{j+1}$ modulo the action of the symmetric group $S_{j+1}$ by permutation of the points. So our first step is devoted to some intersection theory on symmetric products of elliptic curves.

\subsection{Symmetric products}
Let $C$ be a smooth complete curve of genus $1$. To understand the Chow ring of $C^{(n)}$, we identify the $n$-fold symmetric product $C^{(n)}$ with the set of degree $n$ divisors on $C$ and consider the map $u: C^{(n)}\to\textrm{Pic}^n( C)$ that sends a divisor $D$ to its divisor class $[D]$. This is a projective space bundle (see e.g.~\cite[D.5.2]{3264}): The fiber over $[D]$ is $\pp(H^0(C,\sL(D))\cong \P^{n-1}$. Hence the Chow ring $A(C^{(n)})$ is a free $A(\textrm{Pic}^n( C))$-module of rank $n$ via $u^*$, see \cite[Section~9.3]{3264}. For any $P\in C$ we denote by $x_P\in A^1(C^{(n)})$ the class of the set $X_P=\{D\in C^{(n)}:\, D-P\geq0\}$. Let $D=P_1+\ldots+P_n$ for some $P_i\in C$. Then $$x_{P_1}\cdots x_{P_n}=\{D\}=x_{P_1}\cdots x_{P_{n-1}}\cdot u^*[D]$$ and in particular $x_{P}^n-x_{P}^{n-1}\cdot u^*[n P]=0$ for any $P\in C$. In fact, the class $x_P$ generates $A(C^{(n)})$ as a $A(\textrm{Pic}^n( C))$-algebra and its minimal polynomial is given by $x_{P}^n-x_{P}^{n-1}\cdot u^*[n P]=0$. All classes of the form $x_P$ are numerically equivalent and we denote their numerical equivalence class by $x$. Furthermore, we denote by $\theta$ the numerical equivalence class of $u^*[n P]$ which is just the fiber of $u$ over a point. Thus the ring of cycles on $C^{(n)}$ modulo numerical equivalence is given by $$\Z[\theta,x]/(\theta^2,x^n-\theta x^{n-1}).$$

\begin{ex}
 If $n=2$, then $X = C^{(2)}$ is a ruled surface as discussed in \cite[Chapter~V, Section 2]{Hart77}. We have $\theta=f$ (a fiber) and $x=C_0$ (a section) with the notation in \cite[Proposition~V.2.3]{Hart77}. The intersection product on this surface satisfies $\theta.x = f.C_0 = 1$, $\theta^2 = f^2 = 0$. This already determines the central surface invariants $p_a(X) = -1$, $p_g(X) = 0$, $q(X) = 1$ in \cite[Corollary~V.2.5]{Hart77}.

 It can be checked that we have $x^2 = C_0^2 = 1$ on this surface (which is easy with the description of the Chow ring above) so that $e = 1$ and the canonical class of $X$ is numerically equivalent to $-2C_0 + f$ by \cite[Corollary~V.2.11]{Hart77}.
\end{ex}

To determine the canonical class of $C^{(n)}$ for larger $n$, we consult \cite{ACGHI}.
The Chern polynomial of the tangent sheaf $\cT$ of $C^{(n)}$ is given by (\cite[Chapter~VII, (5.4)]{ACGHI}) 
\begin{align*}
  c_t(\cT)& =(1+tx)^n\cdot e^{-\frac{t\theta}{1+tx}} = (1+tx)^n \left(1-\frac{t\theta}{1+tx}\right) 
  \\
  & =(1+tx)^{n-1}\cdot(1+tx-t\theta).
\end{align*}
In particular, the numerical equivalence class of the canonical divisor is $$-c_1(\cT)=-nx+\theta.$$

We now use this to compute several dimensions needed below.
\begin{lem}\label{lem:amplenum}
  Any line bundle whose numerical equivalence class is $a\theta +b x$ with $a\geq0$ and $b>0$ is ample.
\end{lem}

\begin{proof}
  By \cite[p.~310]{ACGHI}, $x$ is ample. Further $\theta$ is nef as the pullback of a nef divisor \cite[Exp.~1.4.4]{nefl}. Now the claim follows from \cite[Cor.~1.4.10]{nefl}
\end{proof}

\begin{cor}\label{cor:kodaira}
  Any line bundle $\sL$ on $C^{(n)}$ whose numerical equivalence class is $a\theta +b x$ with $a\geq1$ and $b>-n$ satisfies $h^i(\sL)=0$ for all $i>0$.
\end{cor}

\begin{proof}
  We have seen that the numerical equivalence class of the canonical bundle $\Omega$ on $C^{(n)}$ is $-nx+\theta$. Therefore, by \Cref{lem:amplenum}, the line bundle $\sL\otimes\Omega^{-1}$ is ample. Thus the claim follows from the Kodaira Vanishing Theorem \cite[Rem.~7.15]{Hart77}.
\end{proof}

Since we have factored the Chern polynomial, the Todd class of the tangent sheaf $\cT$ on $C^{(n)}$ is (\cite[Appendix~A, Section~4]{Hart77})
\[
  \td(\cT)=\left(\frac{x}{1-e^{-x}}\right)^{n-1}\cdot \frac{x-\theta}{1-e^{\theta-x}}.
\]

\begin{prop}\label{prop:eulercomp}
  The Euler characteristic of a line bundle $\sL$ on $C^{(n)}$ that is numerically equivalent to $a\theta + bx$ with $b\geq 0$ is 
  \[
    \chi(\sL) = \binom{b+n}{n} + (a-1)\binom{b+n-1}{n-1}.
  \]
\end{prop}

\begin{proof}
  Since $\theta^2 = 0$ in $A(C^{(n)})$, the exponential Chern character of $\sL$ is numerically equivalent to $(1+a\theta)e^{bx}$. So by the Hirzebruch--Riemann--Roch Theorem \cite[Theorem~A.4.1]{Hart77}, we have
  \[
    \chi(\sL)=\deg\left((1+a\theta)\cdot e^{bx}\cdot \left(\frac{x}{1-e^{-x}}\right)^{n-1}\cdot \frac{x-\theta}{1-e^{\theta-x}}\right)_n.
  \]
  We first simplify the term $\frac{x-\theta}{1-e^{\theta-x}}$, again using the fact that $\theta^2=0$ so that
  \begin{align*}
    \frac{x-\theta}{1-e^{\theta-x}} & = \sum_{m=0}^n \frac{B_m}{m!}(x-\theta)^m = \sum_{m=0}^n \frac{B_m}{m!}(x^m - m\theta x^{m-1}) \\
    & =  \frac{x}{1-e^{-x}} - \theta \sum_{m=1}^n \frac{B_m}{(m-1)!}x^{m-1},
  \end{align*}
  where $B_m$ denotes the appropriate Bernoulli number. By this computation, we have to compute three terms (again, because $\theta^2 = 0$)
  \begin{align}\label{eq:threeterms}
    \chi(\sL) = &\deg\left(e^{bx}\left(\frac{x}{1-e^{-x}}\right)^n\right)_n - \nonumber \\
    & \deg\left(\theta e^{bx}\left(\frac{x}{1-e^{-x}}\right)^{n-1}\sum_{m=0}^{n-1} \frac{B_{m+1}}{(m)!} x^m \right)_n + \\
    & a \deg\left(\theta e^{bx}\left(\frac{x}{1-e^{-x}}\right)^n\right)_n. \nonumber
  \end{align}
  We can compute this degree using the Todd class of the tangent bundle on $\P^{n-1}$, which is
  \[
    \td(\cT_{\P^{n-1}}) = \left( \frac{x}{1-e^{-x}} \right)^{n} \in \Z[x]/(x^n) \otimes \Q,
  \]
  because our variable $x$ satisfies almost the same relation.
  More precisely, the important property is that the ideal $(\theta^2,x^n-\theta x^{n-1})$ is homogeneous.  
  The Hirzebruch--Riemann--Roch Theorem on $\P^m$ implies that
  \[
    \deg\left(e^{dx}\cdot \left( \frac{x}{1-e^{-x}} \right)^{m+1}\right)_m = \chi(\sO_{\P^m}(d)) = \binom{d+m}{m}.
  \]
  The third term $\deg(\theta e^{bx} (x/(1-e^{-x}))^n)_n$ in the above \Cref{eq:threeterms} is the coefficient of $\theta x^{n-1}$, which is given by the coefficient of $x^{n-1}$ in $e^{bx} (x/(1-e^{-x}))$ and by the Hirzebruch--Riemann--Roch formula of $\P^{n-1}$ is
  \[
    \deg\left( e^{bx}\cdot \left(\frac{x}{1-e^{-x}}\right)^{n}\right)_{n-1} = \chi(\sO_{\P^{n-1}}(b)) = \binom{b+n-1}{n-1}.
  \]
  To compute the first two terms in \Cref{eq:threeterms}, we use the following identity
  \[
    \frac{\rm d}{{\rm d} x} \left(e^{bx}\left(\frac{x}{1-e^{-x}}\right)^n\right) = b e^{bx} \left(\frac{x}{1-e^{-x}}\right)^n + n e^{bx} \left(\frac{x}{1-e^{-x}}\right)^{n-1} \sum_{m=1}^n \frac{B_m}{(m-1)!} x^{m-1},
  \]
  which shows that the first term in \eqref{eq:threeterms} is up to a factor of $1/n$ given by the coefficient of $x^{n-1}$ on the right hand side. This we can almost compute using Hirzebruch--Riemann--Roch on $\P^{n-1}$ again, namely
  \begin{align*}
    & \deg\left(b e^{bx} \left(\frac{x}{1-e^{-x}}\right)^n + n e^{bx} \left(\frac{x}{1-e^{-x}}\right)^{n-1} \sum_{m=1}^n \frac{B_m}{(m-1)!} x^{m-1} \right)_{n-1} =\\
    & b \binom{b+n-1}{n-1} + n \deg \left( e^{bx} \left(\frac{x}{1-e^{-x}}\right)^{n-1} \sum_{m=1}^n \frac{B_m}{(m-1)!} x^{m-1} \right)_{n-1}.
  \end{align*}
  This second summand here cancels with the second term in \Cref{eq:threeterms}. So in total, we get
  \[
    \chi(\sL) = \left(\frac{b}{n} + a\right) \binom{b+n-1}{n-1} = \binom{b+n}{n} + (a-1)\binom{b+n-1}{n-1}. \qedhere
  \]
\end{proof}

\begin{cor}\label{cor:eulercompo}
  Let $\sL$ be a line bundle on $C^{(n)}$ that is numerically equivalent to $a\theta + bx$ with $a\geq1$ and $b\geq 0$. Then 
  \[
    h^0(C^{(n)},\sL)= \binom{b+n}{n} + (a-1)\binom{b+n-1}{n-1}.
  \]
\end{cor}

\begin{proof}
 By \Cref{cor:kodaira} we have $h^0(C^{(n)},\sL)=\chi(\sL)$. Now the claim follows from \Cref{prop:eulercomp}.
\end{proof}

\subsection{The secant varieties of elliptic normal curves}
Let $C\subset \P^n$ be a curve of genus $1$ embedded by a complete linear system $|D|$ with $n\geq 4$. To relate the secant variety $\sigma_j(C)\subset \P^n$ (together with its embedding) to the symmetric product $C^{(j+1)}$, we consider the rational map $\psi\colon \sigma_j(C) \ratto \P^{N-1}$ defined by a set of generators of degree $j+1$ of the homogeneous vanishing ideal of $\sigma_{j-1}(C)$, see \Cref{facts:ellnormal}. The other tool are determinantal representations of elliptic normal curves as explained e.g.~in \cite[Section~6C]{Eis05}:
Let $E$ be a divisor on $C$ of degree at least $2$ and at most $\deg(D)/2$.
Consider the multiplication map
\[
  \phi_E \colon H^0(C,\sL(E)) \times H^0(C,\sL(D-E)) \to H^0(C,\sL(D)).
\]
The representing matrix of size $\deg(E)\times \deg(D-E)$ is $1$-generic and its $2\times 2$ minors cut out $C$ (see \cite[Proposition~6.10]{Eis05}). The maximal ($\deg(E)\times \deg(E)$) minors cut out the variety obtained as the union of the span of all the divisors of zeros of sections in $H^0(C,\sL(E))$ as a subscheme of $C\subset \P^n$ (\cite[Corollary~6.12]{Eis05}). This discussion is crucial for the proof of the following statement.

\begin{prop}\label{prop:ellnormalsymmprod}
  Let $C\subset \P^n$ be an elliptic normal curve and let $k$ be an integer such that $2k+1\leq n$. Let $N=\binom{n-k}{k+1}+\binom{n-k-1}{k}$ and $\psi\colon \sigma_k(C) \ratto \P^{N-1}$ be the rational map defined by the $(k+1)$-ics vanishing on $\sigma_{k-1}(C)$. This map factors through a map $C^{(k+1)}\to\P^{N-1}$. The fiber of a point $p$ in the image of $\psi$ is a $k$-dimensional subspace $|E|$ for a unique divisor $E$ on $C$ of degree $k+1$. The divisor class on $C^{(k+1)}$ determined by this map is numerically equivalent to $2\theta+(n-(2k+1)) x$. More precisely, it is given by $$u^*([E]+[E'])+\sum_{k=1}^{n-(2k+1)} x_{P_k}$$ where $E$ and $E'$ are divisors of degree $k+1$ on $C$ and $P_k\in C$ such that $E+E'+\sum_{k=1}^{n-(2k+1)} {P_k}$ is linearly equivalent to $D$.
\end{prop}

\begin{proof}
 By \cite[Thm. 1.2]{tommyfisher}, the number $N$ is exactly the dimension of the space of $(k+1)$-ics vanishing on $\sigma_{k-1}(C)$. For every such $(k+1)$-ic $q$ and $v_0,\ldots,v_k\in \C^{n+1}$ with $[v_i]\in C$ we consider $\widetilde{q}=q(x_0\cdot v_0+\cdots+x_k\cdot v_k)$. Since $q$ vanishes on $\sigma_{k-1}(C)$, the polynomial $\widetilde{q}$ vanishes on every coordinate hyperplane in the span of $\{v_0,v_1,\ldots,v_k\}$ and hence on every coordinate subspace. Therefore, the only monomial that shows up in $\widetilde{q}$ is $x_0\cdots x_k$. Its coefficient is a polynomial that is multilinear in the $v_i$. Thus we obtain a morphism $\varphi:C^{k+1}\to\P^{N-1}$ with the property that $\varphi(p_0,\ldots,p_k)=\psi(x)$ for every $x\not\in\sigma_{k-1}(C)$ which lies in the $k$-secant spanned by $p_0,\ldots,p_k$. In particular, since $\varphi$ does not depend on the order of the arguments, we obtain our desired map $C^{(k+1)}\to\P^{N-1}$.
 
 In order to determine the corresponding divisor on $C^{(k+1)}$ (i.e.~the hyperplane section of $C^{(k+1)}$), we need to describe the $k$-secants in the zero set of a $(k+1)$-ic that vanishes on $\sigma_{k-1}(C)$. First we consider the case $2k+1= n$. Let $E$ be a divisor of degree $k+1$ on $C$. Then $\phi_E$ from the above discussion (\cite[Section~6C]{Eis95}) is represented by a square matrix of size $k+1$ and its determinant $q$ vanishes on $\sigma_{k-1}(C)$ since the matrix has rank one on $C$. The zero set of $q$ is the union of the span of all the effective divisors linearly equivalent to $E$. This corresponds to the divisor $u^*[E]$ on $C^{(k+1)}$. By symmetry, the polynomial $q$ also vanishes on the union of all the divisors linearly equivalent to $D-E$. We claim that $q$ vanishes on no more $k$-secants, which shows that the divisor on $C^{(k+1)}$ we are looking for is $u^*([E]+[D-E])$. To this end let $S$ be any $k$-secant on which $q$ vanishes, spanned by a divisor $F$ not linearly equivalent to $E$. Then there is a $k$-secant $S'$ spanned by an effective divisor $F'$ linearly equivalent to $E$ that intersects $S$. Thus the span of $F+F'$ is a hyperplane which implies that $F$ is linearly equivalent to $D-E$.
 
 For the case $2k+1<n$ we pick points $P_1,\ldots,P_{n-(2k+1)}$ on $C$ and look at their span $W$. We project $C$ linearly to $\P^{2k+1}$ from the center $W$ so that the image is again an elliptic normal curve $\widetilde{C}$. Every $(k+1)$-ic vanishing on $\sigma_{k-1}(\widetilde{C})$ also vanishes on all $k$-secants of $C$ that contain one of the points $P_i$ as well as on the preimages of the $k$-secants of $\widetilde{C}$ it vanishes on and that have already been determined.
\end{proof}

\begin{rem}
  Again, the above proof is particularly nice from a geometric point of view for an elliptic normal curve $C\subset \P^4$ of degree $5$.
  In this case, to show that the image of the map $\psi$ is isomorphic to the symmetric product $C^{(2)}$, we can argue using Hirzebruch surfaces (rational normal scrolls), see \cite[Exercise~A2.22]{Eis95}. Choose any divisor class $E = P_1 + P_2$ of degree $2$. Then $h^0(\sO(E)) = 2$ and $h^0(\sO(-E)\otimes \sL) = 3$.
  The multiplication map 
  \[
    H^0(C,\sO(E))\times H^0(C,\sO(-E)\otimes \sL) \to H^0(C,\sL)
  \]
  shows that $C$ is contained in a rational normal scroll of type $S_{2,1}\cong \mathscr{H}_1$ and the lines of the ruling of the scroll intersect the curve $C$ in divisors that are linearly equivalent to $E$. In other words, the ruling lines of the scroll are a $\pic^2(C)$-family of secants to $C$. So for every divisor class in $E\in \pic^2(C)\cong C$, we get a rational normal scroll $S_E$ containing $C$. Additionally, $S_E\cap S_{E'} = C$ if $E$ is not linearly equivalent to $E'$.
  The secant variety $\Sigma = \sigma_1(C)$ is therefore the union $\bigcup_{E\in\pic^2(C)}S_E$.

  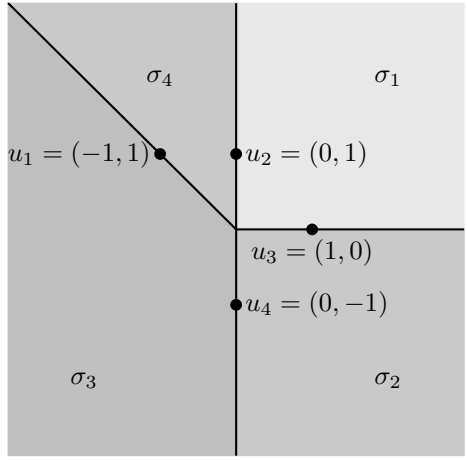
\begin{figure}
    \begin{center}
      \begin{tikzpicture}
        \filldraw[black!10!white!90] (0,0) -- (0,3) -- (3,3) -- (3,0) -- (0,0);
        \filldraw[black!30!white!70] (0,0) -- (3,0) -- (3,-3) -- (0,-3) -- (0,0);
        \filldraw[black!50!white!50] (0,0) -- (0,-3) -- (-3,-3) -- (-3,3) -- (0,0);
        \filldraw[black!70!white!30] (0,0) -- (0,3) -- (-3,3) -- (0,0);
        \draw[thick] (0,0) -- (0,3);
        \draw[thick] (0,0) -- (3,0);
        \draw[thick] (0,0) -- (0,-3);
        \draw[thick] (0,0) -- (-3,3);
        \filldraw (0,1) circle (2pt) node[right] {$u_2 = (0,1)$};
        \filldraw (1,0) circle (2pt) node[below] {$u_3 = (1,0)$};
        \filldraw (0,-1) circle (2pt) node[right] {$u_4 = (0,-1)$};
        \filldraw (-1,1) circle (2pt) node[left] {$u_1=(-1,1)$};
        \draw (2,2) node {$\sigma_1$};
        \draw (2,-2) node {$\sigma_2$};
        \draw (-2,-2) node {$\sigma_3$};
        \draw (-1,2) node {$\sigma_4$};
      \end{tikzpicture}
    \end{center}
    \caption{A picture of the rational polyhedral fan defining the Hirzebruch surface $\mathscr{H}_1$ as a toric variety.}
    \label{fig:hirzebruch}
  \end{figure}

  To study the rational map above, we first restrict it to the rational normal scrolls $S_E$ with $E\in\pic^2(C)$. These are embeddings of $\mathscr{H}_1$ into $\P^4$ via complete linear systems that are linearly equivalent to the torus invariant Weil divisor $H = D_3 + D_4$ (following the notation and conventions in \cite{CLSMR2810322} as shown in \Cref{fig:hirzebruch}). A line $F$ of the ruling of $S_E$ is linearly equivalent to $D_3$. So the class of the curve in $\pic(S_E)$ is $2H - F = D_3 + 2 D_4$, which is the anticanonical class of $S_E$.
  Here we again see by computing the intersection product $C.F = (2H-F).F = 2$, that $S_E$ is a union of secants of $C$.

  In $S_E$, the elliptic curve $C$ is cut out by two quadrics $q_0^E,q_1^E$ (whereas the ideal of $S_E$ in degree $2$ is a $3$-dimensional subspace of $\sI(C)_2$). The fibers of the rational map $Q_E\colon S_E\ratto \P^1$, $p\mapsto (q_0^E(p):q_1^E(p))$, are cut out by a quadric on $S_E$ (because $Q_E^{-1}(\lambda:\mu) = \{p\in S_E\colon \lambda q_1^E(p) -\mu q_0^E(p) = 0\}$). In $\pic(S_E)$, this is the class of $2H$ minus the class of the curve $C$, which leaves $F$, the class of a line of the ruling of $S_E$. Since the map $Q_E$ is constant along secants to $C$, the fiber of a point $(\lambda:\mu)$ is exactly one line of the ruling of $S_E$.
  Putting this together for every $E\in \pic^2(C)$, this shows that the fibers of the above rational map $\Sigma\ratto \P^4$ are exactly secants of $C$ so that the Zariski closure of the image is a surface $X\subset \P^4$.

  This construction also gives us a map from $C\times C$ to $X$, which takes equal values for $(P,Q)$ and $(Q,P)$. Hence it induces a morphism from $C^{(2)}$ to $X$. This is an isomorphism.
  In fact, $X$ is an embedding of $C^{(2)}$. The class of the corresponding line bundle on $X$ in $\pic(C^{(2)}) = \Z + \pi_P^*(\pic C)$ is $C_0 + \pi_P^*(A+B)$ for some $A,B\in C$. 
  This line bundle is very ample by \cite[Exercise~V.2.12]{Hart77}.
\end{rem}

\begin{lem}\label{lem:noquadratics}
  Let $C\subset \P^{2k+2}$ be an elliptic normal curve for some integer $k$. There are no quadratic relations among the generating set of $2k+3$ forms of degree $k+1$ of the vanishing ideal of $\sigma_{k-1}(C)$.
\end{lem}

\begin{proof}
  Let $f_0,\ldots,f_{2k+2}$ be a generating set of $2k+3$ forms of degree $k+1$ of the vanishing ideal of $\sigma_{k-1}(C)$. Let $S$ be the polynomial ring in three variables over $\C$. It suffices to show that the restrictions $g_i \in S$ of the $f_i$ to a generic two-plane do not satisfy any quadratic relations. Since $\sigma_{k-1}(C)$ is arithmetically Gorenstein by \cite[Cor.~8.14]{hulekthedogsout}, the ideal $I\subset S$ generated by the $g_i$ is Gorenstein as well (and zero dimensional). By the Buchsbaum--Eisenbud Theorem \cite{eisenbuchs}, see also \cite[Thm.~3.4.1]{Bruns93}, there exists a free $S$-module $F$ of rank $2k+3$ and an alternating homomorphism $\varphi:F\to F^*$ of degree $1$ such that $I$ is generated by the $(2k+2)\times (2k+2)$ Pfaffians of $\varphi$. Moreover, the minimal free resolution of $S/I$ is given by
  \[
    0\to S(-2k-3)\to F\cong S(-k-2)^{2k+3}\xrightarrow{\varphi} F^* \cong S(-k-1)^{2k+3}\to S.
  \]
  Now by \cite{powerofpfaff} the minimal free resolution of $S/I^2$ for such $I$ is of the form 
  \[
    0\to\wedge^{2k+1}F^* \to G \to \Sym^2(F^*)\cong S(-2k-2)^N\to S,
  \]
  where $N = \binom{2k+4}{2}$. Since this is a minimal free resolution, there are no linear relations among the pairwise products of Pfaffians generating $F^*$. In other words, there are no quadratic relations among the $g_i$.
\end{proof}

\begin{lem}\label{lem:symprodprojnorm}
  Let $C\subset \P^{n}$ be an elliptic normal curve and $2k+1\leq n$. Consider the map $C^{(k+1)}\to\P^{N-1}$ whose corresponding divisor $D$ we described in \Cref{prop:ellnormalsymmprod}. The linear system on $C^{(k+1)}$ that is given by all linear forms on $\P^{N-1}$ is complete. If $n=2k+2$, then the linear system on $C^{(k+1)}$ that is given by all quadratic forms on $\P^{N-1}$ is complete as well.
\end{lem}

\begin{proof}
 The divisor $D$ is numerically equivalent to $2\theta+(n-2k-1)x$ by \Cref{prop:ellnormalsymmprod}. Thus by \Cref{cor:eulercompo} the associated complete linear system has dimension $\binom{n-k}{k+1} + \binom{n-k-1}{k}=N$. Because there are no $(k+1)$-ics vanishing on $\sigma_k(C)$, see \Cref{facts:ellnormal}, the linear system on $C^{(k+1)}$ that is given by all linear forms on $\P^{N-1}$ also has dimension $N$, and thus is complete.

 For the quadratic case with $n=2k+2$, we follow the same argument: Here, $\sigma_k(C)$ is a hypersurface of degree $2k+3$ and $N = 2k+3$, see \Cref{facts:ellnormal}. So there are no $(2k+2)$-ics vanishing on $\sigma_k(C)$ and by \Cref{lem:noquadratics}, the linear system on $C^{(k+1)}$ that is given by all quadratic forms on $\P^{N-1}$ has dimension
 \[
   \frac{1}{2}\cdot N\cdot(N+1)= (2k+3)\cdot(k+2).
 \]
 But by \Cref{cor:eulercompo} the complete linear system $|2D|$ also has dimension
 \[
   h^0(C^{(k+1)},4\theta + 2 x) = \binom{k+3}{k+1} + (4-1)\binom{k+2}{k} = (2k+3)\cdot(k+2).\qedhere
 \]
\end{proof}

\subsection{Constructing symmetric determinantal representations}
\begin{lem}\label{lem:Hilfslemma}
  Let $C$ be an irreducible smooth curve and let $\Delta\subset \C$ be the open unit disk. 
  Let $s_0\colon \Delta \to C^{(k+1)}$ be an analytic arc such that the support of $s_0(0)$ consists of $k+1$ distinct points, where we think of an element of $C^{(k+1)}$ as an effective divisor of degree $k+1$ on $C$. Furthermore, fix $D_0\in C^{(k)}$ such that $s_0(0) - D_0$ is effective. Then there is an $\epsilon >0$ and there are analytic arcs $s_1\colon \epsilon\Delta \to C^{(k)}$ and $s_2\colon \epsilon\Delta\to C$ such that for all $t\in \epsilon\Delta$ we have $s_0(t) = s_1(t) + s_2(t)$ and $s_1(0) = D_0$.
\end{lem}

\begin{proof}
  Consider the map $C^{(k)}\times C \to C^{(k+1)}$, $(E,Q)\mapsto E+Q$. This map is unramified over $s_0(0)$ and therefore locally an unramified cover. So the claim follows by the implicit function theorem.
\end{proof}

\begin{lem}\label{lem:ellnormaltangentcone}
  Let $C\subset \P^{2k+2}$ be an elliptic normal $M$-curve.
  \begin{enumerate}
    \item The hypersurface $X = \sigma_k(C)$ is hyperbolic with respect to a suitable point $e\in \P^{2k+2}(\R)$.
    \item The degree of $X$ is $2k+3$.
    \item The hypersurface $X$ has multiplicity $3$ at every point that lies on a $(k-1)$-secant of $C$ but not on a $(k-2)$-secant.
  \end{enumerate}
\end{lem}

\begin{proof}
  The first two parts are shown above: $(1)$ follows from \Cref{cor:tollhyp} and the fact that any divisor of odd degree on an elliptic $M$-curve is {\placeholdertwo}; $(2)$ is covered by \Cref{facts:ellnormal}. We prove the third claim by studying the tangent cone of $X$ at a general point on a general $(k-1)$-secant of $C$. So let $p$ be a smooth point on $\sigma_{k-1}(C)$, which means that $p$ does not lie on any $(k-2)$-secant. We claim that the tangent cone of $X$ at $p$ is the join of $C$ with the tangent space to $\sigma_{k-1}(C)$ at $p$. 

  The tangent cone $TC_p\sigma_k(C)$ is the union of all lines that are in limiting position for secant lines $\overline{pq}$ with $q\in \sigma_k(C)$ as $q$ approaches $p$, see \cite[Exercise~20.4]{Ha95}. So we clearly have $T_p\sigma_{k-1}(C)\subset TC_p\sigma_k(C)$. Let $\gamma\colon \Delta \to \sigma_{k-1}(C)$ be an analytic arc, where $\Delta$ is the open unit disk in $\C$, with $\gamma(0) = p$ and $\gamma'(0) = q \in T_p\sigma_{k-1}(C)$ and let $[x]\in C$ be a point with a fixed affine representative $x$. For every $(1:s)\in\P^1$, consider the map $\wt{\gamma}_s\colon \Delta \to \sigma_k(C)$, $t\mapsto \gamma(t) + s t x$. Clearly, $\wt{\gamma}(0) = p$ and $\wt{\gamma}'(0) = q + s x$. Since $TC_p\sigma_k(C)$ is Zariski-closed, this shows one inclusion. For the other inclusion, let $D_0$ be the divisor corresponding to the unique $(k-1)$-secant containing $p$. We consider an analytic arc $q\colon \Delta\to \sigma_k(C)$ with $q(0) = p$. We further assume that for all $t\in \Delta\setminus\{0\}$, the image $q(t)$ does not lie on any $(k-1)$-secant of $C$ and denote the corresponding divisor in $C^{(k+1)}$ by $E_t$. Then $E_0 - D_0$ is effective. Assuming genericity of the arc, the support of $E_0$ consists of distinct points. By \Cref{lem:Hilfslemma}, there is an $\epsilon>0$ and there are analytic arcs $s_1\colon \epsilon\Delta\to C^{(k)}$ and $s_2\colon \epsilon\Delta\to C$ such that for all $t\in \epsilon\Delta$ we have $E_t = s_1(t) + s_2(t)$ and $s_1(0) = D_0$. 
  The secant line to $\sigma_k(C)$ spanned by $q(t)$ and $p$ lies in the span of $s_2(t)$ and the line spanned by $p$ and a point in the span of $s_1(t)$.
  So the limiting position of such a secant line lies in the span of $s_2(0)$ and a tangent line to $\sigma_{k-1}(C)$ at $p$. 
  This proves our claim that the tangent cone is the join of the tangent space and the curve.

  This tangent cone has degree $3$, because we get a curve of degree $2k+3 - (2k) = 3$ in $\P^2$ when we project the curve $C$ from $T_p \sigma_{k-1}(C)$. Indeed, the tangent space $T_p \sigma_{k-1}(C)$ is the span of the tangent spaces $T_{x_i}C$ to the curve at the points $x_1,x_2,\ldots,x_k$ that span the $(k-1)$-secant containing $p$ by Terracini's Lemma \cite[Proposition~4.3.2]{JoinsAndIntersections}. So the projection of $C$ with center $T_p \sigma_{k-1}(C)$ corresponds to the linear system of the divisor $H.C - 2(x_1+x_2+\ldots+x_k)$, where $H.C$ is the hyperplane section. The image has therefore degree $3$.
\end{proof}

In the following proof, we use the well-established theory of multiplicities in the context of hyperbolic polynomials, see e.g.~\cite{RenMR2198215}. So let $f\in\R[x_0,x_1,\ldots,x_n]$ be homogeneous of degree $d$ and hyperbolic with respect to $e\in\R^{n+1}$. The \emph{multiplicity} of $x$ with respect to $f$ and $e$ is the multiplicity of the root $t=0$ of the polynomial $p(t) = f(x + te)$. It turns out that the multiplicity of a point $x$ does not depend on the chosen element $e$ in the hyperbolicity cone of $f$. Moreover, this shows that the multiplicity of a point is equal to the multiplicity of the hypersurface $f$ at the point $x\in\R^{n+1}\setminus\{0\}$. In particular, the multiplicity of a real point $x$ is $1$ if and only if it is a smooth real point on the hypersurface $\sV(f)$. If $g$ is an interlacing polynomial of $f$ (with respect to $e$), then the multiplicity of $x$ with respect to $g$ and $e$ is at least the multiplicity of $x$ with respect to $f$ and $e$ minus $1$ for all $x\in \sV(f)(\R)$. This claim is the observation that a polynomial $g\in\R[x_0,x_1]$ can only interlace a polynomial $f\in\R[x_0,x_1]$ if the multiplicity of $g$ at every multiple root of $f$ is at least the multiplicity of $f$ at this root minus $1$.
\begin{lem}\label{lem:signint}
  Let $C\subset \P^{2k+2}$ be an elliptic normal $M$-curve with $C(\R) = S_0 \cup S_1$, $S_i$ homeomorphic to $\P^1(\R)$, and let $f$ be a polynomial defining the hypersurface $X = \sigma_k(C)$. Fix $e\in\R^{2k+3}$ such that $f$ is hyperbolic with respect the $e$. Let $g$ be an interlacer of $f$.
  \begin{enumerate}
    \item The polynomial $g$ has even degree $2(k+1)$.
    \item The polynomial $g$ is semidefinite on every $(k-1)$-secant of $C$.
  \end{enumerate}
\end{lem}

\begin{proof}
  The first claim follows from the fact that $\deg(X) = 2k+3$. To show the second part, we use that the multiplicity of the interlacer $g$ at a point $x\in X$ is at least the multiplicity of $f$ at this point minus $1$, which implies that $g$ vanishes with multiplicity at least $2$ in every $(k-1)$-secant of $C$ by the previous Lemma~\ref{lem:ellnormaltangentcone}. Let $S = \langle p_0,p_1,\ldots,p_k\rangle$ be a general $k$-secant of $C$. Then $g$ vanishes to order at least two along the $k+1$ hyperplanes $\langle p_i\colon i\in \{0,1,\ldots,k\}\setminus\{j\}\rangle$ ($j=0,1,\ldots,k$) of $S$. Since the degree of $g$ is $2k+2$, it follows that $g$ restricted to $S$ is, up to scaling, the product of the squares of the linear forms defining these hyperplanes. In particular, it is positive or negative semidefinite on $S$. Since this holds on general secants, it follows for all of them.
\end{proof}

We again consider the rational map $\psi\colon \P^{2k+2}\ratto \P^{N-1}$ defined by the $N = \beta(k+1,2k+3) = \binom{2k+3 - (k+1)}{k+1} + \binom{2k+3 - (k+1)-1}{k} = 2k+3$ homogeneous forms of degree $k+1$ that generate the vanishing ideal of $\sigma_{k-1}(C)$ (see \cite[Theorem~1.2]{tommyfisher}). The image of this map is $C^{(k+1)}$ embedded into $\P^{2k+2}$ via a line bundle that is numerically equivalent to $2\theta+x$ (in the Chow ring of $C^{(k+1)}$ modulo numerical equivalence), see~\Cref{prop:ellnormalsymmprod}.
\begin{lem}\label{lem:ellnormalconnectivity}
  Let $C\subset \P^{2k+2}$ be an elliptic normal $M$-curve so that $C(\R) = S_0\cup S_1$ with $S_i$ homeomorphic to $\P^1(\R)$. 
  Then $C^{(j+1)}(\R) = M_0\cup M_1$ has two connected components $M_i$ for every $j\geq 0$.
\end{lem}

\begin{proof}
  The symmetric product $C^{(j+1)}(\R)$ has two connected components, because it is a $\P^j$-bundle over $C$, which is an $M$-curve of genus $1$. 
\end{proof}

\begin{rem}
  Identifying elements of $C^{(j+1)}$ with divisors of degree $j+1$ on $C$, the two connected components of $C^{(j+1)}(\R)$ are the sets of conjugation invariant divisors on $C$ with different parity of real points on $S_0$, where $C(\R) = S_0\cup S_1$ with $S_i$ homeomorphic to $\P^1(\R)$. More explicitly, if $j+1$ is even, then one connected component of $C^{(j+1)}(\R)$ contains all conjugation invariant divisors $D$ such that the number (counted with multiplicity) of points in the support of $D$ that are in $S_0$ is even (and therefore also the number of points in $\supp(D)\cap S_1$). The other connected component contains the divisors with $\supp(D)\cap S_i$ odd (for $i=0,1$). In case that $j+1$ is odd, the two cases are that $\supp(D)\cap S_0$ is even (and hence $\supp(D)\cap S_1$ odd) or that $\supp(D)\cap S_0$ is odd (and hence $\supp(D)\cap S_1$ even). 
\end{rem}

Since the rational map $\psi\colon \sigma_k(C)\ratto \P^{N-1}$ is continuous on the set $\sigma_k(C)(\R)\setminus \sigma_{k-1}(C)(\R)$, its connected components either map to $M_0$ or $M_1$. 
\begin{lem}\label{lem:ellnormalinterlacerchar}
  Let $C\subset\P^{2k+2}$ be an elliptic normal $M$-curve and let $X = \sigma_k(C)$. Let $f$ be a polynomial with $\sV(f) = X$ and $g$ a homogeneous polynomial of degree $2k+2$. Then $g$ is an interlacer of $f$ if and only if there is a unique quadric hypersurface $Y = \sV(q)\subset \P^{N-1}$ such that $\psi^{-1}(Y) = \sV(g)$, where the quadric $q$ has constant but opposite signs on the two connected components of $C^{(k+1)}$.
\end{lem}

\begin{proof}
  Let $g$ be an interlacer of $f$ with the property that $g$ does not vanish on any point in $\sigma_k(C)(\R)\setminus \sigma_{k-1}(C)(\R)$, e.g.~a directional derivative of $f$ with
  respect to $e$ in the hyperbolicity cone of $f$. Let $L\cong \P^1$ be a general line through $e$ so that the roots of $f$ and $g$ alternate along $L$ (viewed from $e$). Since $g$ does not vanish at any point of $\sigma_k(C)(\R)\setminus \sigma_{k-1}(C)(\R)$, it has constant sign on the connected components of this set.
  Since $g$ vanishes on $\sigma_{k-1}(C)$ and has degree $2k+2$, there is a unique quadratic hypersurface $Y = \sV(q)\subset \P^{N-1}$ with $\sV(g) = \psi^{-1}(Y)$, which has constant sign on $M_0$ and $M_1$. 
  The fact that $g$ interlaces $f$ implies that the connected component of $C^{(k+1)}(\R)$ on which the image of the intersection point of $L$ and $X$ under $\psi$ lies alternates along $L$, viewed from $e$.
  This quadric $q$ must therefore have opposite signs on the two connected components of $C^{(k+1)}$. 
  
  More generally, if $g_0$ is any interlacer, we get the claim by a limit argument considering $g_0 + \epsilon g$ for $\epsilon \to 0$.
  
  On the other hand, if a defining polynomial $q$ of a quadratic hypersurface $Y\subset\P^{N-1}$ has different signs on the two connected components of $C^{(k+1)}(\R)$, then $h = \psi^*(q)$ has different signs on the connected components of $\sigma_k(C)(\R)\setminus \sigma_{k-1}(C)(\R)$ mapping to $M_0$, resp.~$M_1$. So for a directional derivative $g = D_e f$ as above, $gh$ is nonnegative on $\sigma_k(C)(\R)$ (after changing sign, if necessary). 
  Now \cite[Theorem~2.1]{MR3395549} implies that $h$ is an interlacer of $f$.
\end{proof}

\begin{lem}\label{lem:contactint}
  There is a quadratic form $q$ which has constant but opposite signs on the two connected components of $C^{(k+1)}$ and whose zero divisor on $C^{(k+1)}$ is of the form $2\cdot ( u^*([B]+[B'])+x_{P_0})$ for some $P_0\in C$ and divisors $B,B'$ on $C$ of degree $k+1$.
\end{lem}

\begin{proof}
  By \Cref{prop:ellnormalsymmprod} and \Cref{lem:symprodprojnorm} there is a hyperplane $H$ whose intersection divisor with $C^{(k+1)}$ is given by $u^*([E_1]+[E_2])+x_P$ for some $P\in C$ divisors $E_1,E_2$ of degree $k+1$ on $C$. We consider the embedding $C\hookrightarrow C^{(k+1)}, \,Q\mapsto Q+kQ_0 $ for some $Q_0\in C$ and denote its image by $C_0$. The intersection of $H$ with $C_0$ is then given by $D=Q_1+Q_2+P$ where $Q_i$ is the unique point on $C$ such that $Q_i+kQ_0\equiv E_i$. We consider the embedding of $C$ to $\pp^2$ given by $D$. We can describe the image in Weierstra\ss{} form: $$y^2=\prod_{i=1}^3(x-p_i)$$ for some real numbers $p_1<p_2<p_3$. We denote by $P_0$ the point at infinity and by $P_1$ and $P_2$ the points $(p_1,0)$ and $(p_2,0)$. Letting $l_i=x-p_i$ we have that the zero divisor of $l_1l_2$ is given by $2\cdot(P_0+P_1+P_2)$ and is linearly equivalent to $2D$. Now we consider the divisor $A=2\cdot u^*([P_1+kQ_0]+[P_2+kQ_0])+2x_{P_0}$ on $C^{(k+1)}$. Again by \Cref{prop:ellnormalsymmprod} and since $$2(P_1+kQ_0)+2(P_2+kQ_0)+2P_0\equiv 2Q_1+2Q_2+2P+4kQ_0\equiv 2E_1+2E_2+2P,$$ we find that $A$ is linearly equivalent to $2u^*([E_1]+[E_2])+2x_P$. By \Cref{lem:symprodprojnorm} there is thus a quadratic form $q$ whose zero divisor is $A$. Since all zeros of $q$ are of even multiplicity, it has constant sign on the two connected components of $C^{(k+1)}$. In order to examine these signs it suffices to consider the restriction of $q$ to $C_0$ because $C_0$ is a section of the projective bundle $C^{(k+1)}\to C$. By construction the zero divisor of $q$ on $C_0$ is given by $2\cdot(P_0+P_1+P_2)$ and thus $q|_{C_0}$ is some scalar multiple of $l_1l_2$ which indeed has opposite signs on each of the connected components of $C(\R)$.
\end{proof}

\begin{rem}
  The line bundle associated to the divisor $u^*([B]+[B'])+x_{P_0}$ from the previous lemma is an Ulrich line bundle on $C^{(n)}$. Indeed, by construction it satisfies the conditions of \cite[Prop.~4.1(ii)]{beauvilleUlrich}
\end{rem}

\begin{thm}\label{thm:ellnormaldetrep}
  Let $C\subset \P^{2k+2}$ be an elliptic normal $M$-curve of degree $2k+3$. There exists a definite determinantal representation for polynomial $f$ generating the vanishing ideal of the hyperbolic hypersurface $\sigma_k(C)$.
\end{thm}

\begin{proof}
  We start with the quadratic form $q$ from \Cref{lem:contactint} whose divisor is of the form $2D$ with $D= u^*([B]+[B'])+x_{P_0}$ for some $P_0\in C$ and divisors $B,B'$ on $C$ of degree $k+1$. By \Cref{cor:eulercompo} and \Cref{lem:symprodprojnorm} we can find linearly independent quadratic forms $h_{11},h_{12},\ldots,h_{1(2k+3)}$ that vanish on $D$ and $h_{11}=q$. We complete the $(2k+3)\times (2k+3)$ matrix
  \[
\begin{pmatrix}
  h_{11} & h_{12} & h_{13} & \cdots & h_{1(2k+3)} \\
  h_{12} & & & & \\
  h_{13} & & & & \\
  \vdots & & & & \\
  h_{1(2k+3)} & & & &
\end{pmatrix}
  \]
  to a matrix of rank $1$ on $C^{(k+1)}$
  by finding $h_{ij}$ with the property that $h_{11}h_{ij} = h_{1i}h_{1j}\in H^0(C^{(k+1)},4D)$. In fact, this $h_{ij}$ is the unique section of $H(C^{(2)},\sL(2D))$ whose zero divisor is $(h_{1i})_0 + (h_{1j})_0-2D$ (which is linearly equivalent to $2D$). By construction, this matrix has rank $1$ on $C^{(k+1)}$ wherever $h_{11}$ is non-zero. By semi-continuity of the rank, it has rank $1$ everywhere on $C^{(k+1)}$.
  Next, we pull this matrix back to $\sigma_k(C)$ along $\psi$, that is we set $a_{ij} = \psi^*(h_{ij})$. The matrix $A = (a_{ij})$ has rank $1$ along $\sigma_k(C)$ and it is filled with forms of degree $2k+2$. 
  The entry $a_{11}$ is an interlacer of $f$ by \Cref{lem:ellnormalinterlacerchar}. So \cite[Theorem~4.7]{MR3395549} implies our claim.
\end{proof}

\begin{ex}\label{ex:ellparticularlynice}
  To construct a particularly nice determinantal representation of the hyperbolic hypersurface $\Sigma=\sV(f)=\sigma_1(\iota(C))$ from \Cref{ex:elliptic1}, we pick the divisor $D'$ in the construction given in the above proof by pushing the following quadrics forward through the rational map $\psi\colon \Sigma \ratto \P^4$, $p\mapsto (q_0(p):q_1(p):\ldots:q_4(p))$, where the quadrics $q_i$ are a generating set for the homogeneous vanishing ideal of $\iota(C)\subset \P^4$. Take the cone over the curve $\iota(C)$ with vertex $\iota(1:0:1) = (1:0:1:0:1)\in\P^4$, which is on the curve. The saturated ideal of this surface is generated by two quadrics. This pencil of quadrics contains $4$ singular quadrics (since this is a pencil of quadrics in four variables after projecting away from the vertex). We choose the following two out of these four.
  \begin{align*}
    f_1 = &\left(\sqrt{2}+1\right)z_0^2+\left(-\sqrt{2}-1\right)z_0z_2+z_1z_2-z_0z_3+\left(-\sqrt{2}-1\right)z_1z_3+z_2z_3+ \\
     & \left(\sqrt{2}+1\right)z_3^2+\left(-\sqrt{2}-1\right)z_0z_4-z_1z_4+\left(\sqrt{2}+1\right)z_2z_4 \\
    f_2 = &\left(-\sqrt{2}+1\right)z_0^2+\left(\sqrt{2}-1\right)z_0z_2+z_1z_2-z_0z_3+\left(\sqrt{2}-1\right)z_1z_3+z_2z_3+\\
    & \left(-\sqrt{2}+1\right)z_3^2+\left(\sqrt{2}-1\right)z_0z_4-z_1z_4+\left(-\sqrt{2}+1\right)z_2z_4
  \end{align*}
  These two quadrics are pull backs of linear forms via $\psi$ since they vanish on the curve $\iota(C)$; say $f_i = \psi^*(\ell_i)$. The divisors of these linear forms on $\psi(\Sigma)\cong C^{(2)}$ are of the form $C_0 + 2P_i f$ and $P_2 - P_1 = T$ is a $2$-torsion point on $C$. 
  So now we set $D' = C_0 + P_1f + P_2f$. Pulling this divisor back to $\Sigma$ and taking the minimal free resolution of $\psi^*(2D')$ gives a definite determinantal representation of $\Sigma$. Indeed, the computation shows
  \[
    f = \det 
    \begin{pmatrix}
      {-{z}_{2}}&{z}_{0}&{-{z}_{1}}&{-{z}_{1}}&{-{z}_{3}}\\
      {z}_{0}&{-{z}_{2}}&0&{z}_{3}&0\\
      {-{z}_{1}}&0&{z}_{0}&0&{z}_{2}\\
      {-{z}_{1}}&{z}_{3}&0&-{z}_{0}+{z}_{4}&0\\
      {-{z}_{3}}&0&{z}_{2}&0&{z}_{4}
    \end{pmatrix}
    \]
    and this matrix is definite at the point $(2:0:-3:0:6)$.
    This linear positive definite determinantal representation certifies that $\Sigma$ is a hyperbolic hypersurface.
\end{ex}

\section{Hyperbolic and Spectrahedral Shadows}\label{sec:extendedform}

\subsection{$M$-curves}
In \cite{clauscurves} Scheiderer proved that the closed convex hull of any one-dimensional semi-algebraic set is a spectrahedral shadow. However, his proof does not give any bounds on the size of the matrices or on the number of additional slack variables. The purpose of this section is to show that the convex hull of a connected component of an $M$-curve can be represented as a \emph{hyperbolic shadow}. We further give degree and dimension bounds that can be expressed only in terms of the genus and the degree of the $M$-curve in question.

\begin{thm}\label{thm:hyperbolicshadows}
  Let $C\subset\pp^n$ be an $M$-curve of degree $d$ and genus $g$ and let $C_0$ be a very compact connected component of $C(\R)$. The convex hull of $C_0$ is a shadow of an at most $d+1$-dimensional hyperbolicity cone  of degree at most 
  \[
    \left(\left\lfloor\frac{d}{2}\right\rfloor+1\right) \cdot\sum_{i=0}^{\lfloor\frac{d}{2}\rfloor}\binom{g}{i}-g\cdot\sum_{i=0}^{\lfloor\frac{d}{2}\rfloor-1}\binom{g-1}{i} 
  \]
  If $g\leq\lfloor\frac{d}{2}\rfloor$, then this can be bounded from above by $(d+2-g)\cdot 2^{g-1}$.
\end{thm}

\begin{proof}
  Let $H$ be the divisor on $C$ given by a hyperplane section. There is an effective divisor $D$ of degree at most $g$ such that $H+D$ has even degree on $C_0$ and odd degree on all the other connected components of $C(\R)$ because $C_0$ is very compact (see~\Cref{rem:verycompact}. Clearly, the divisor class $[H+D]$ is {\placeholdertwo}. We let $d'$ be the degree of $H+D$ and notice that $H+D$ is nonspecial by Lemma \ref{lem:tollnonspecial}. Let $\sL=\sL(H+D)$ be the line bundle corresponding to this divisor class and $\iota(C)\hookrightarrow\pp^{d'-g}$ the corresponding embedding. Then there are suitable sections $s_0,\ldots,s_n\in\Gamma(C,\sL)$ that vanish on $D$ such that the corresponding linear projection $\pi:\pp^{d'-g}\dashrightarrow\pp^n$ maps $\iota(C)$ to $C$. By Proposition \ref{prop:hypconetoll}, the convex hull of $\iota(C_0)$ is the hyperbolicity cone of $\sigma_k(C)$ where $k=\frac{d'-g}{2}-1$. The image of this hyperbolicity cone under $\pi$ is the convex hull of $C_0$. Since we have $k\leq\lfloor\frac{d}{2}\rfloor-1$ and since the formula in Corollary \ref{cor:deg} for the degree of $\sigma_k(C)$ is monotonically increasing in $k$, we get the upper bound of
  \[
  \left(\left\lfloor\frac{d}{2}\right\rfloor+1\right) \cdot\sum_{i=0}^{\lfloor\frac{d}{2}\rfloor}\binom{g}{i}-g\cdot\sum_{i=0}^{\lfloor\frac{d}{2}\rfloor-1}\binom{g-1}{i} 
  \]
  for the degree of $\sigma_k(C)$. The last part follows as in Corollary \ref{cor:deg}.
\end{proof}

\begin{rem}
  Note that for fixed genus $g$, both the dimension and the degree of the hyperbolicity cone in Theorem \ref{thm:hyperbolicshadows} grow linearly in the degree $d$ of the curve. If $d$ is fixed, then the dimension of the hyperbolicity cone is constant and the degree grows polynomially in the genus $g$ of the curve.
\end{rem}

\begin{rem}
  For $g=0$ the bound in \Cref{thm:hyperbolicshadows} is given by $\lfloor\frac{d}{2}\rfloor+1$. Moreover, by \Cref{ex:hankel} the hyperbolicity cone that we project is a spectrahedron. This shows that the convex hull of a rational curve of even degree is a shadow of an at most $d+1$-dimensional spectrahedral cone that can be described by matrices of size at most $\lfloor\frac{d}{2}\rfloor+1$. This fact has first been proved in \cite{dieter0}. Using our results from \Cref{sec:elliptic} we will next prove a similar result for elliptic curves.
\end{rem}

\subsection{Elliptic curves}
For elliptic curves, the situation is better understood. In \cite{clausg1} Scheiderer shows that the convex hull of any elliptic curve, whose real part is compact, is a spectrahedral shadow. Moreover, in this case he gives upper bounds on the size of the involved matrices. However, these bounds depend on the specific curve and they even tend to infinity when the curve is degenerated to a singular one \cite[Cor.~4.6]{clausg1}. He further remarks that his bounds are in a sense optimal when constructing a representation as spectrahedral shadow using the Lasserre relaxation \cite[Rem.~2.15]{clausg1}. Using the results from the previous sections, we can reprove that the closed convex hull of an elliptic curve, whose real part is compact, is a spectrahedral shadow and our proof yields bounds that only depend (linearly) on the degree of the elliptic curve.

\begin{prop}\label{prop:shadowellipticM}
  Let $C\subset\pp^n$ be an elliptic $M$-curve of degree $d$ and let $C_0$ be a very compact connected component of $C(\R)$. The closed convex hull of $C_0$ is a shadow of an at most $(d+1)$-dimensional spectrahedral cone that can be described by matrices of size at most $d+1$.
\end{prop}

\begin{proof}
 By \Cref{thm:ellnormaldetrep} the hyperbolic polynomial from \Cref{thm:hyperbolicshadows} that realizes the closed convex hull of $C_0$ as a hyperbolic shadow actually has a definite determinantal representation.
\end{proof}

\begin{lem}\label{lem:elliticdoublecover}
  Let $C$ be an elliptic curve such that $C(\R)$ is connected and nonempty. There is an elliptic $M$-curve $\tilde{C}$ together with an unramified double cover $\tilde{C}\to C$.
\end{lem}

\begin{proof}
  We can assume that $C$ is given in Weierstra\ss{} normal form $$y^2=x\cdot(x^2+ax+b)$$for some $a,b\in\R$ with $b\neq0$ and $4b>a^2$. We let $\tilde{C}$ be the double cover of $\pp^1$ ramified at two pairs of complex conjugate points defined by $w^2=t^4+at^2+b$. By Hurwitz's Theorem $\tilde{C}$ has genus $1$. Since the double cover $\tilde{C}\to\pp^1,\,(t,w)\mapsto t$ is real fibered, the curve $\tilde{C}$ is of dividing type and therefore an elliptic $M$-curve. Finally, the map defined by $\tilde{C}\to C,\,(t,w)\mapsto(t^2,tw)$ has degree two and is unramified by Hurwitz's Theorem.
\end{proof}

\begin{thm}\label{thm:ellipticshadow}
  Let $C\subset\pp^n$ be an elliptic curve of degree $d$ and let $C_0$ be a very compact connected component of $C(\R)$. The convex hull of $C_0$ is a shadow of an at most $(2d+1)$-dimensional spectrahedral cone that can be described by matrices of size at most $2d+1$.
\end{thm}

\begin{proof}
  If $C$ is an $M$-curve, the claim follows from Proposition \ref{prop:shadowellipticM}.
  If $C$ is not an $M$-curve, then we can find an elliptic $M$-curve $\tilde{C}$ and an unramified double cover $f:\tilde{C}\to C$ as in Lemma \ref{lem:elliticdoublecover}. The line bundle $\sL=f^*\cO_C(1)$ on $\tilde{C}$ has degree $2d$. For $0\leq i\leq n$ we let $s_i=f^* x_i$ where $x_0,\ldots,x_n$ are the homogeneous coordinates of $\pp^n$ and we complete $s_0,\ldots,s_n$ to a basis $s_0,\ldots,s_{2d-1}$ of $\Gamma(\tilde{C},\sL)$. We embed $\tilde{C}$ to $\pp^{2d-1}$ via $s_0,\ldots,s_{2d-1}$. We have now realized $f$ as the restriction to $\tilde{C}$ of the linear projection $\pp^{2d-1}\dashrightarrow\pp^n$ onto the first $n+1$ coordinates. Since $f$ is unramified, each connected component $\tilde{C}_0$ of $\tilde{C}(\R)$ is mapped surjectively onto $C_0$. Thus since the convex hull of $\tilde{C}_0$ is the projection of a $(2d+1)$-dimensional spectrahedron of size $2d+1$ by Proposition \ref{prop:shadowellipticM}, the same is true for the convex hull of $C_0$. 
\end{proof}

\begin{ex}
  We apply \Cref{thm:ellipticshadow} to an elliptic normal $M$-curve in $\P^3$. For this, we embed the cubic curve $C\subset \P^2$ defined by the equation $x_0^3 - x_0 x_2^2 - x_1^2 x_2$ as in \Cref{ex:elliptic1} via the linear system $\sL(4P)$, where $P = (0:1:0)$. This linear system gives the map
  \[
    \iota\colon \left\{
      \begin{array}{l}
        C \ratto \P^3 \\
        (x_0:x_1:x_2) \mapsto (x_0^2:x_0x_2:x_1x_2:x_2^2).
      \end{array}
      \right\}
  \]
  The vanishing ideal of the image is generated by $z_0z_3-z_1^2$ and $z_0z_1 - z_1z_3 - z_2^2$ in coordinates $(z_0:z_1:z_2:z_3)$ on $\P^3$. This embedding is a coordinate projection of the embedding of $C$ to $\P^4$ given in \Cref{ex:elliptic1} from $(0:1:0:0:0)$. In $\P^3(\R)$, the curve has two very compact connected components. We take the convex hull of the component $S_0$ that is the image of the very compact component of $C(\R)\subset \P^2(\R)$. By \Cref{ex:ellparticularlynice}, this convex hull has the semidefinite representation
  \[
    \begin{array}{rl}
      \conv(S_0) = & \{(z_0:z_1:z_2:z_3)\in\P^3(\R) \, \vert \, \exists t\in \R \colon \\
      & \begin{pmatrix}
          -z_1 & z_0 & -t & -t & -z_2 \\
          z_0 & -z_1 & 0 & z_2 & 0 \\
          -t & 0 & z_0 & 0 & z_1 \\
          -t & z_2 & 0 & -z_0 + z_3 & 0 \\
          -z_2 & 0 & z_1 & 0 & z_3
        \end{pmatrix} \text{ is semidefinite} \}.
    \end{array}
  \]
  Note that by \Cref{ex:ellconv3} it has even a spectrahedral representation.
\end{ex}

\begin{rem}
  We discuss the relation of the previous \Cref{thm:ellipticshadow} to various results obtained by Scheiderer. Scheiderer constructed semidefinite representations for convex hulls of curves relying on sums of squares methods (namely on stability of preorders on curves) in \cite{clauscurves}. He has worked out the case of elliptic curves more explicitly in \cite{clausg1}. He shows that he cannot obtain bounds for the size of the semidefinite representation that only depend on the degree of the embedding of the genus $1$ curve. Using his methods, the bounds on the size of the representation correspond to degree bounds for sum-of-squares representations. Our result, on the other hand, gives a bound on the size only in terms of the degree, namely $2d+1$. First of all, this of course means, that the representation that we construct is not directly obtained by Scheiderer's sum-of-squares methods. On the other hand, in light of \cite[Theorem~3.3]{clausschatten}, the semidefinite representation corresponds to some morphism $\phi\colon X \to C$ such that there are degree bounds for sum-of-squares representations for the pullbacks of the linear system embedding $C$ on $X$.
\end{rem}

\section{Open Questions}
We would like point the reader to a few open questions arising from our work.
\begin{enumerate}
  \item To show the existence of vastly real linear systems, we construct them as maximally odd linear systems, using the interlacing property. How to find vastly real linear systems that are not maximally odd? 
  \item Can the rigid isotopy results from \cite{mikhalkin2019rigid} be generalized to vastly real linear systems of dimension $2k+1$ for $k>1$?
  \item Do vastly real linear systems exist on real curves of genus $g$ with fewer than $g+1$ connected components exist? \Cref{rem:nonm} suggests that the answer might be \emph{no}.
  \item The convex hulls of very compact connected components of $M$-curves embedded by a vastly real linear system of degree $2k+2+g$ are hyperbolicity cones. By Scheiderer's results \cite{clauscurves}, they have semidefinite representations. But are they spectrahedra, as the Generalized Lax Conjecture claims?
\end{enumerate}

\bigskip

\noindent \textbf{Acknowledgements.}
Part of this work was done while the authors were visiting the Simons Institute for the Theory of Computing. We would like to thank Kristian Ranestad for helpful discussions.

\def\cprime{$'$}

\end{document}

%% file: operators.tex
\usepackage{mathrsfs}

\newcommand\wt{\widetilde}

\DeclareMathOperator\td{td}
\DeclareMathOperator\pic{Pic}

\newcommand\sO{\mathscr O}
\newcommand\sI{\mathscr I}
\newcommand\ratto{\dashrightarrow}
\renewcommand\P{\mathbb{P}}

\newcommand{\conv}{\operatorname{conv}}

\newcommand{\Sym}{\operatorname{Sym}}

\newcommand{\Hom}{\operatorname{Hom}} 
\newcommand{\Spec}{\operatorname{Spec}}

\newcommand{\Gr}{\mathbb{G}}
\newcommand{\codim}{\operatorname{codim}}

\newcommand{\adj}{\operatorname{adj}} 

\DeclareMathOperator\cone{cone}

\DeclareMathOperator{\cent}{cent}

%% file: fonts.tex
\newcommand{\cF}{{\mathcal F}}

\newcommand{\cO}{{\mathcal O}}

\newcommand{\cT}{{\mathcal T}}

\newcommand{\cV}{{\mathcal V}}

\newcommand{\fD}{{\mathfrak D}}

\newcommand{\fL}{{\mathfrak L}}

\newcommand{\A}{{\mathbb A}}
\newcommand{\G}{{\mathbb G}}
\newcommand{\C}{{\mathbb C}}
\newcommand{\R}{{\mathbb R}}
\newcommand{\pp}{\mathbb{P}}
\newcommand{\Q}{{\mathbb Q}}

\newcommand{\Z}{{\mathbb Z}}

